\documentclass[reqno]{amsart}
\usepackage{amsmath,amssymb,amsthm}
\usepackage{mathtools}
\usepackage{tikz}
\usepackage[hidelinks,unicode]{hyperref}

\newtheorem{theorem}{Theorem}
\newtheorem{lemma}[theorem]{Lemma}
\newtheorem{prop}[theorem]{Proposition}
\theoremstyle{definition}
\numberwithin{equation}{section}

\newcommand{\eps}{\varepsilon}
\newcommand{\Dv}{\D\ve}
\newcommand{\Gd}{\Gamma_{\rm D}}
\newcommand{\Gn}{\Gamma_{\rm N}}
\newcommand{\Gf}{\Gamma_{\rm F}}

\newcommand{\f}[2]{\frac{#1}{#2}}
\newcommand{\dd}[1]{\,\mathrm{d}{#1}}
\newcommand{\embl}{\hookrightarrow}
\newcommand{\ii}{\int_{\Omega}}
\newcommand{\inn}{\int_{\Gn}}
\newcommand{\nn}{\nabla}
\newcommand{\wc}{\rightharpoonup}
\newcommand{\Sym}{\R^{3\times3}_{\rm sym}}

\newcommand{\bb}{\mathbb}
\newcommand{\R}{\bb R}
\newcommand{\N}{\bb N}

\newcommand{\D}{\bb D}
\newcommand{\Sb}{\bb S}
\newcommand{\I}{\bb I}
\newcommand{\T}{\bb T}

\newcommand{\Cl}{\mathcal{C}}
\newcommand{\Ul}{\mathcal{U}}
\newcommand{\Vl}{\mathcal{V}}
\newcommand{\X}{\mathcal{X}}

\newcommand{\pr}{p}
\newcommand{\qr}{q}

\newcommand{\bs}{\boldsymbol}
\newcommand{\ve}{\bs v}
\newcommand{\we}{\bs w}
\newcommand{\uu}{\bs u}
\newcommand{\n}{\bs n}

\newcommand{\z}{\bs z}
\newcommand{\fe}{\bs f}
\newcommand{\gb}{\bs g}
\newcommand{\fit}{\bs\varphi}
\newcommand{\pit}{\bs\psi}
\newcommand{\etb}{\bs\xi}
\newcommand{\s}{\bs s}

\newcommand{\UUU}{}
\newcommand{\EEE}{}
\DeclarePairedDelimiter{\norm}{\lVert}{\rVert}
\DeclarePairedDelimiter{\scal}{\langle}{\rangle}
\DeclareMathOperator{\di}{div}
\DeclareMathOperator{\rot}{rot}
\DeclareMathOperator*{\esssup}{ess\,sup}

\usepackage{accents}
\renewcommand*{\dot}[1]{\accentset{\bs\bullet}{#1}}
\newcommand{\doc}[1]{\accentset{\bs\circ}{#1}}

\title[Variational resolution of the Navier-Stokes with
outflows]{Variational resolution of outflow boundary conditions for \UUU incompressible Navier-Stokes}
\author[M. Bathory]{Michal Bathory}
\address[Michal Bathory]{Faculty of Mathematics, University of
	Vienna, Oskar-Morgenstern-Platz 1, A-1090 Vienna, Austria}
\email{michal.bathory@univie.ac.at}
\urladdr{http://www.mat.univie.ac.at/$\sim$bathory}
\author[U. Stefanelli]{Ulisse Stefanelli}
\address[Ulisse Stefanelli]{Faculty of Mathematics, University of
  Vienna, Oskar-Morgenstern-Platz 1, A-1090 Vienna, Austria,
Vienna Research Platform on Accelerating
  Photoreaction Discovery, University of Vienna, W\"ahringerstra\ss e 17, 1090 Wien, Austria,
 \& Istituto di
  Matematica Applicata e Tecnologie Informatiche {\it E. Magenes}, via
  Ferrata 1, I-27100 Pavia, Italy
}
\email{ulisse.stefanelli@univie.ac.at}
\urladdr{http://www.mat.univie.ac.at/$\sim$stefanelli}

\begin{document}
	
\begin{abstract}
  \UUU This paper focuses on the so-called Weighted Inertia-Dissipation-Energy (WIDE) variational approach for the approximation of unsteady Leray-Hopf solutions of the incompressible Navier-Stokes system. Initiated in \cite{OSU}, this variational method is here extended to the case of non-Newtonian fluids with power-law index $r\geq11/5$ in three space dimensions and large nonhomogeneous data. Moreover, boundary conditions are not imposed on some parts of boundaries, representing, e.g., outflows. Correspondingly, natural boundary conditions arise from the minimization. In particular, at walls  we recover boundary conditions of Navier-slip type. At outflows and inflows, we obtain the condition $-\f12|\ve|^2\n+\T\n=0$. This provides the first theoretical explanation for the onset of such boundary conditions.
\EEE

	\end{abstract}
	
	\maketitle
	
	\section{Introduction}

        \UUU In this work, we are interested in a variational
        resolution technique of the incompressible Navier-Stokes
        system by means of the so-called {\it Weighted
        Inertia-Dissipation-Energy} (WIDE) functional approach \cite{Serra-Tilli10,dg}.

We \EEE consider \UUU the \EEE unsteady flow of an incompressible
fluid through a~generalized channel $\Omega$ with inlets $\Gd$, walls
$\Gn$, and outlets $\Gf^i$, $i=0,\ldots,n$, as depicted on
Figure~\ref{fig}. The flow is modeled in the bulk by the \UUU
incompressible \EEE Navier-Stokes system
	\begin{equation}
		\di\ve=0,\qquad\partial_t\ve+\ve\cdot\nn\ve-\di\Sb(\Dv)+\nn\pr=\fe,\label{NS0}
	\end{equation}
	where $\Sb$ describes the constitutive relation between the
        Cauchy stress tensor and the symmetric velocity gradient
        $\Dv\coloneqq\f12(\nn\ve+(\nn\ve)^T)$.

\UUU The main focus of this paper is to \EEE show \UUU that weak solutions \EEE to this system
can be obtained \UUU as limits \EEE as $\eps\to0_+$ of \UUU minimizers
\EEE of the {\it WIDE} functionals
	\begin{align}
		I_{\eps}(\ve)&=\int_0^{\infty} e^{-\frac{t}{\eps}}\ii\Big(\f{\eps}2|\doc{\ve}|^2-\fe\cdot\ve\Big)\dd{t}\nonumber\\
		&\quad+\int_0^{\infty} e^{-\frac{t}{\eps}}\Big(\ii\int_0^1\Sb_{\eps}(\lambda\Dv)\cdot\Dv\dd{\lambda}+\inn\int_0^1\s_{\eps}(\lambda\ve)\cdot\ve\dd{\lambda}\Big)\dd{t},\label{I0}
	\end{align}
	when minimized over \UUU whole \EEE trajectories satisfying \textit{only} the constraints
	\begin{equation}\label{constrs}
		\di\ve=0,\quad\ve\vert_{\Gd}=\UUU \ve_{\rm D}, \EEE \quad\ve\vert_{\Gn}\cdot\n=0,\quad\int_{\Gf^i}\ve\cdot\n=F_i,
	\end{equation}
	for some given (inflow) data $\ve_{\rm D}$ and some given net flux rates $F_i$ through $\Gf^i$, $i=0,1,\ldots,n$.
\begin{figure}\label{fig}
	\centering
	\tikzset{every picture/.style={line width=0.75pt}}       
	\begin{tikzpicture}[x=0.75pt,y=0.75pt,yscale=-0.8,xscale=0.7]
		\draw [line width=1] (80,106) .. controls (248,104) and (307,183) .. (306,46) ;
		\draw [line width=1] (350,46) .. controls (330,130) and (330,149) .. (435,46) ;
		\draw [line width=1] (240,226) .. controls (340,124) and (370,165) .. (387,226) ; 
		\draw [line width=1] (485,46) .. controls (329,131) and (561,210) .. (610,180) ;
		\draw [line width=1] (80,156) .. controls (278,158) and (291,157) .. (215,226) ;
		\draw [line width=1] (424,226) .. controls (387,131) and (561,240) .. (610,210) ;
		\draw [dash pattern={on 0.84pt off 2.51pt}]  (306,46) -- (350,46) ;
		\draw [dash pattern={on 0.84pt off 2.51pt}]  (80,106) -- (80,156) ;
		\draw [dash pattern={on 0.84pt off 2.51pt}]  (436,46) -- (485,46) ;
		\draw [dash pattern={on 0.84pt off 2.51pt}]  (215,226) -- (240,226) ;
		\draw [dash pattern={on 0.84pt off 2.51pt}]  (610,180) -- (610,210) ;
		\draw [dash pattern={on 0.84pt off 2.51pt}]  (387,226) -- (424,226) ;
		\draw [dash pattern={on 4.5pt off 3pt}] (80,156) .. controls (162,139) and (162,123) .. (80,106) ;
		\draw (318,130) node [anchor=north west][inner sep=0.75pt] {$\Omega$};
		\draw (320,22) node [anchor=north west][inner sep=0.75pt] {$\Gf^0$};
		\draw (218,230) node [anchor=north west][inner sep=0.75pt] {$\Gf^0$};
		\draw (615,185) node [anchor=north west][inner sep=0.75pt] {$\Gf^0$};
		\draw (50,123) node [anchor=north west][inner sep=0.75pt] {$\Gd$};
		\draw (452,22) node [anchor=north west][inner sep=0.75pt] {$\Gf^1$};
		\draw (395,230) node [anchor=north west][inner sep=0.75pt] {$\Gf^2$};
		\draw (90,60) node [anchor=north west][inner sep=0.75pt][align=left] {solid lines $\Gn$: $\ve\cdot\n=0$};
		\draw (55,165) node [anchor=north west][inner sep=0.75pt] {$\ve=\ve_{\rm D}$};
		\draw (487,35) node [anchor=north west][inner sep=0.75pt] {$\int_{\Gf^1}\ve\cdot\n=F_1$};
		\draw (425,218) node [anchor=north west][inner sep=0.75pt] {$\int_{\Gf^2}\ve\cdot\n=F_2$};	
	\end{tikzpicture}
	\caption{A two-dimensional example of the domain $\Omega$.}
\end{figure}
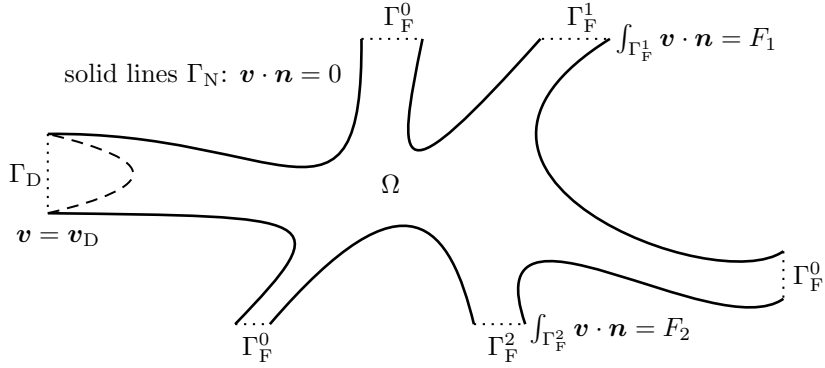
	The symbol $\doc{\ve}$ denotes a~kind of material derivative defined by
	\begin{equation}
		\doc{\uu}\coloneqq\partial_t\uu+\rot\uu\times\uu=\partial_t\uu+\uu\cdot\nn\uu-\nn(\tfrac12|\uu|^2).\label{dotu}
              \end{equation}
     \UUU Here, \EEE $\rot\ve=\nn\times\ve$, i.e.,
     $(\rot\ve)_i=\sum_{j,k=1}^3\eps_{ijk}\partial_j\ve_k$, where
     $\eps_{ijk}$ is the Levi-Civita symbol. Further, the symbol $\n$
     is the outward normal vector on $\partial\Omega$.
\UUU In fact, \UUU $\doc{\uu} $ \EEE can be replaced by the usual material derivative
	$	\dot{\uu}\coloneqq\partial_t\uu+\uu\cdot\nn\uu$
              \UUU in some specific cases, as explained below. \EEE
	Moreover, the function $\s$ is an analogue of $\Sb$ on the
        boundary, modeling the friction forces on $\Gn$. The functions
        $\Sb_{\eps}$ and $\s_{\eps}$ approximate $\Sb$ and $\s$,
        respectively, by improving their asymptotic growth.

        \UUU The relation between the
        minimization of the WIDE functional $I_\eps$ and the Navier-Stokes system is
        revealed by {\it formally} computing the Euler-Lagrange
        equation of $I_\eps$, that is, by assuming smoothness. Let
        $\ve_{\eps}$ minimize $I_\eps$ in some reasonable set of trajectories
        obeying \eqref{constrs} and compute the Gateaux
        derivative of $I_{\eps}$ with respect to $\ve$ in a~direction
        $\fit$, which \textit{does not have fixed boundary values on
          $\Gn$ and $\Gf$}. \UUU Via \EEE  integration by parts (more precisely, the Stokes theorem) this leads to 
	\begin{align}
		&\doc{\ve}_{\eps}-\eps\partial_t\doc{\ve}_{\eps}+\eps\doc{\ve}_{\eps}\times\rot\ve_{\eps}+\eps\rot(\ve_{\eps}\times\doc{\ve}_{\eps})-\di\Sb_{\eps}(\Dv_{\eps})+\nn\qr_{\eps}\nonumber\\
		&\quad+(\Sb_{\eps}(\Dv_{\eps})\n+\s_{\eps}(\ve_{\eps}))_{\tau}\vert_{\Gn}+(\eps\ve_{\eps}\times\doc{\ve}_{\eps}\times\n-\qr_{\eps}\n+\Sb_{\eps}(\Dv_{\eps})\n)\vert_{\Gf}=\fe,\label{EL0}
	\end{align}
	where $\qr_{\eps}$ is the Lagrange multiplier corresponding to
        the constraint $\di\ve_{\eps}=0$. \UUU  Here,
the subscript ${\tau}$ denotes \UUU the \EEE tangential part of
        a~vector on $\partial\Omega$, i.e.,
        $\we_{\tau}\coloneqq\we-(\we\cdot\n)\n=\n\times\we\times\n$.
        
\UUU
        System \eqref{EL0}
        is of second order in time. In fact, the occurrence of
        the term $-\eps\partial_t\doc{\ve}_{\eps}$ shows the {\it
          elliptic-in-time} character of \eqref{EL0}. \EEE 
        \UUU Formally taking \EEE the limit $\eps\to0_+$ then leads to
	\begin{equation}\label{gr0}
		\doc{\ve}+\nn\qr-\di\Sb(\Dv)+(\Sb(\Dv)\n+\s(\ve))_{\tau}\vert_{\Gn}+(-\qr\n+\Sb(\Dv)\n)\vert_{\Gf}=\fe,
	\end{equation}
	which, using \eqref{dotu}, $\di\ve=0$, and defining $\pr\coloneqq\qr+\f12|\ve|^2$, can be rewritten as
	\begin{align}\label{gr}
		&\partial_t\ve+\ve\cdot\nn\ve+\nn\pr-\di\Sb(\Dv)\nonumber\\
		&\qquad+(\Sb(\Dv)\n+\s(\ve))_{\tau}\vert_{\Gn}+(-\pr\n-\tfrac12|\ve|^2\n+\Sb(\Dv)\n)\vert_{\Gf}=\fe.
	\end{align}
	Thus, \UUU for $\eps \to 0_+$ we \EEE recover the Navier-Stokes equation
	\begin{equation}\label{NS}
		\partial_t\ve+\ve\cdot\nn\ve+\nn\pr-\di\Sb(\Dv)=\fe
	\end{equation}
	\UUU in \EEE $\Omega$.
        In addition, identity \eqref{gr} \UUU delivers the following \EEE information on the boundary:
	\begin{equation*}
		(\Sb(\Dv)\n+\s(\ve))_{\tau}\vert_{\Gn}+(-\pr\n-\tfrac12|\ve|^2\n+\Sb(\Dv)\n)\vert_{\Gf}=0.
	\end{equation*}
\UUU These, together with the forced boundary conditions
\eqref{constrs}, can be rewritten as \EEE
	\begin{align}
		\ve&=\ve_{\rm D}\;\qquad\text{on }\Gd,\label{ac0}\\
		(\s(\ve)+\Sb(\Dv)\n)_{\tau}=0,\qquad\quad\;\;\,\ve\cdot\n&=0\quad\qquad\text{on }\Gn,\label{ac1}\\
		-\pr\n-\tfrac12|\ve|^2\n+\Sb(\Dv)\n=c_i\n,\quad\int_{\Gf^i}\ve\cdot\n&=F_i\qquad\;\;\text{on }\Gf^i,\label{ac2}
	\end{align}
	for some  constants $c_i$, $i=0,\ldots,n$. Note that, if \EEE  there is more than one outflow, these constants cannot not be fixed a~priori as they are given, e.g., by
	\begin{align}
		c_i&=\f1{|\Gf^i|}\int_{\Gf^i}(-\pr-\tfrac12|\ve|^2+\Sb(\Dv)\n\cdot\n),\quad i=0,\ldots,n,\label{cs}
	\end{align}
	provided that the \UUU integrals are \EEE well-defined. \UUU
        In other words, these   arise as \EEE
       natural boundary condition on $\Gn$ and $\Gf^i$ and \UUU
        are \EEE automatically selected by the minima of $I_{\eps}$.

 \UUU As seen, the minimization of
 $I_\eps$ corresponds to an elliptic-in-time regularization of the
 incompressible Navier-Stokes system. Such regularizations in the
 setting of linear and nonlinear parabolic systems are quite
 classical and have already been considered in
 \cite{Lions:63a,Kohn65,Oleinik}, especially as tool for tackling regularity issues. The reader is
referred to the classical monograph \cite{LM} for an account of
results in the linear setting. The Navier-Stokes system has also been investigated
by elliptic-in-time regularization \cite{Lions:63,Lionsnolinear}, in a
setting however which does not admit a variational structure.   

 The application of the WIDE variational 
 approach to incompressible Navier-Stokes has been initiated in
 \cite{OSU}. There, no-slip boundary conditions are imposed on the
 whole boundary $\partial \Omega$ and the fluid is assumed to be Newtonian. The
 present paper  extends the reach of \cite{OSU} by allowing for in-
 and outlets, by considering general, nonhomogeneous boundary conditions, and by
 allowing non-Newtonian effects in the fluid. \EEE 

 \UUU The possibility of dealing with nonhomogeneous boundary
 conditions, especially with the \EEE outflow boundary condition
 \eqref{ac2}, \UUU is very relevant in relation with applications, \EEE
 see~\cite{Turek} and references therein. \UUU To the best of our
 knowledge, conditions \eqref{ac2} are however \EEE still missing a thorough physical justification. \UUU
 An important contribution of this
 paper is hence that of providing a variational justification of  
 boundary conditions \eqref{ac2}, for they arise as natural conditions
 via minimization of $I_{\eps}$. 

        \UUU The scope of the paper is to make the above formal
        argument rigorous: under
        suitable assumptions on  boundary and initial
          data, we prove that the minimizers of  $I_{\eps}$ over  trajectories
          constrained by \eqref{constrs} converge weakly for
          $\eps\to0_+$ (up to
          subsequences) to a~\UUU {\it Leray-Hopf} solution \cite{Temam} of
          the Navier-Stokes system with boundary conditions
          \eqref{ac0}--\eqref{ac2}.
          As mentioned, the accent is here not on existence, which for
          this system is already known, see \cite{BruneauFabrie,Neustupa2018}, but on variational charaterization of
        natural boundary conditions.

        The structure of the paper is as follows. We collect some
        detail on the physical setting of the problem and on the WIDE
        functional approach in Section
        \ref{sec:phy}.  
        Notation and assumptions are then  presented in Section
        \ref{sec:tech}. The statement of our main result, Theorem
        \ref{Tex}, is in Section \ref{sec:Tex}, where we also record a
        collection of remarks in order to put the statement in
        context. 
        Section \ref{sec:proof} eventually contains the proof of Theorem~\ref{Tex}.

        \EEE

	\section{Physical motivation}\label{sec:phy}
	
	Before we proceed with the rigorous mathematical treatment,
        let us elaborate on the physical meaning of the whole
        procedure and its relation \UUU with \EEE the existing theory.
	
	In Figure~\ref{fig}, the set $\Omega$ is a~representative of
        an open bounded subset of $\R^d$, $d\geq2$, whose Lipschitz
        boundary $\partial\Omega$ is divided into \UUU three different
        types: \EEE 
	\begin{itemize}
		\item[$\Gd$:] A Dirichlet boundary condition is prescribed here. This can model an adhesive boundary (no-slip), or (more importantly) it represents a~prescribed inflow in the nonhomogeneous case.
		\item[$\Gn$:] This set represents the impermeable walls, where $\ve\cdot\n=0$. No condition is imposed in the tangential direction a~priori.
		\item[$\Gf^i$:] The fluid passes freely through this boundary, only the corresponding net flux must be equal to a~given amount, expressed by a~function $F_i$ of time. As such, this represents an artificial boundary, such as outlet or inlet, where nothing is known a~priori about the flow, except for the net flux. Note that due to the incompressibility, it is enough to prescribe $F_i$ for $i=1,\ldots,n$ only, for example.
	\end{itemize}
	The sets $\Gd$, $\Gn$, and $\Gf^i$, $i=0,1,\ldots,n$, are
        open, but \UUU possibly \EEE not connected (in order to
        allow multiple walls, inflows/outflows with shared flux rate
        etc.). However, for simplicity we shall assume that they consist of a~finite number of connected components. Of course, it is possible to simplify this very general setting by omitting certain types of boundaries. We will nevertheless assume that the boundary setting allows for a~corresponding Poincar\'e inequality in $\Omega$.
	
	\subsection{The problem of outflow boundary conditions}
	
	Finding a~reasonable outflow boundary condition for a~flow of
        an incompressible fluid through a~channel is a~longstanding
        problem, which is obviously of great importance in numerics
        and applications. On one hand, some outflow boundary condition
        seems to be required in order to make the problem
        well-posed. On the other hand, usually no a~priori information
        is available at the outflow. Consider for example, the flow in
        a~section of a~pipe. One would formulate the outflow boundary
        condition by trying to replicate the flow behavior as if the
        pipe was a~section of a~much longer pipe, an idealization
        which is often interpreted as the size of the system scaled
        up. \UUU To \EEE this day, there is \UUU still \EEE no
        satisfactory theory dealing with this problem. \UUU Most of
        the \EEE existing mathematical \UUU theories \EEE of
        (incompressible) fluids consider  internal flows only, an assumption that is actually very rarely met in reality. This is in a~sharp contrast with the numerical simulations of the incompressible flow, where several types of outflow boundary conditions are used.
	
	Starting from \cite{Gresho}, where the so-called \textit{do-nothing} boundary condition
	\begin{equation}\label{don}
		-\pr\n+\nu(\nn\ve)\n=0\quad\text{on }\Gf^0,\quad\nu>0,
	\end{equation}
	was introduced for the first time, many outflow boundary
        conditions have been proposed. \UUU All these can be roughly
        divided into \EEE  two classes:
	\begin{itemize}
		\item[(A)] Outflow boundary conditions that are made
                  to fit the geometry of the problem (experiment) \UUU
                  at \EEE hand. This leads to no versatility of application and usually also to poor mathematical properties.
		\item[(B)] Outflow boundary conditions that are guessed, based on the required mathematical properties (such as the validity of the energy estimate, controllable backward flow etc.).
	\end{itemize}
	Although the do-nothing boundary condition \eqref{don} is
        often deemed to be \textit{natural} (since it eliminates the
        whole boundary term in the variational formulation of the
        Navier-Stokes equations), this condition is actually
        a~canonical representative of the class (A). Indeed, this
        condition \UUU is compatible with \EEE  a~Poiseuille flow
        through a~straight channel. This is however the effect of the
        term $(\nn\ve)\n$, \UUU which vanishes \EEE if the outlet is perpendicular to the direction of the flow, which is clearly a~very geometry-dependent property. See \cite{Turek} for some prototypical examples, where \eqref{don} leads to unphysical flows. We also remark that no energy estimate for standard weak solutions of the Navier-Stokes system with \eqref{don} imposed on a~part of the boundary is available, and hence no corresponding large-data, global-in-time existence theory.
	
	On the other hand, the outflow boundary conditions of the class (B) allow to control the unsigned term $|\ve|^2\ve\cdot\n$ on the outlet, hence giving an energy estimate. The outflow boundary conditions
	\begin{equation*}
		-\f12|\ve|^2\n-\pr\n+\nu(\nn\ve)\n=0\quad\text{or}\quad-\frac12\min(0,\ve\cdot\n)\ve-\pr\n-2\nu(\Dv)\n=0,
	\end{equation*}
	studied, e.g., in \cite{Turek}, \cite{BoyerFabrie}, or
        \cite{BraaackMucha}, are typical representatives of the class
        (B). \UUU These \EEE are also particular cases of the class of
        energy-preserving boundary conditions discovered and studied
        in \cite{BruneauFabrie} or more recently in \cite{Ni} using
        a~different approach. \UUU In these works it is also shown
        \EEE  that some of those boundary conditions may produce
        reliable numerical results with respect to experiments, even
        in the case of a~turbulent flow. Unfortunately, the approach
        based on \cite{BruneauFabrie} or \cite{Ni} provides no physical motivation
        for such boundary conditions. This indeed presents a~problem
        since the class (B) is actually huge (basically due to the
        incompressibility constraint, cf.~\cite{Neustupa21}) and
        different outflow boundary conditions can lead to dramatically
        different flow characteristics, at least near the outlet. \UUU
        What is seemingly missing is a specific \EEE selection criterion.
	
	\subsection{An optimization problem \UUU in order to qualify
          boundary conditions}
	
	In this work, we attack the problem of outflow boundary
        conditions by a~rather unorthodox method, that has been so far
        used only in \cite{Bathory2017} to derive the boundary
        conditions of the  do-nothing type for the stationary Stokes
        system. Here, we model the flow by the more appropriate
        unsteady Navier-Stokes equations and we also consider that the
        fluid is non-Newtonian (for example, we allow the dependence
        of the viscosity on the shear rate). The underlying
        mathematical idea is that variational (weak) formulations of
        PDEs (which are also physically more natural) may implicitly
        encode boundary conditions if the test functions have enough
        freedom near the boundary. This leads to looking for
        a~solution in some large space with unspecified boundary
        conditions. It thus seems natural to reformulate everything as
        an optimization problem in such a~space. Choosing an
        appropriate functional $I$ to minimize is, of course,
        a~non-trivial task. On the other hand, in case of the
        dissipative systems such as the Navier-Stokes equations, the
        energy dissipation itself may serve as a~good candidate. In
        fact, this provides a~possible physical explanation and
        a~selection criterion for the obtained (outflow) boundary
        conditions: they are such that the corresponding flow
        dissipates the least amount of energy, in some sense. This
        approach has several other advantages. Firstly, if the
        functional $I$ is indeed related to the physical energy, then
        one \UUU readily \EEE  gets an energy estimate and, \UUU
        consequently, \EEE the existence of a~(weak) solution follows. This is not in contradiction with the lack of an energy estimate for \eqref{don} since we are able to derive \eqref{don} by our method only for certain shear-thickening fluids, cf.~\eqref{don2} below. Another motivation for prescribing the outflow boundary conditions in such an implicit way can be found in \cite{Papa}. There, it is argued that just by applying a~finite-element discretization to the variational formulation of the system with the so-called \textit{no boundary condition}, one implicitly prescribes some outflow boundary conditions that turn out to possess superior numerical properties. See also \cite{Griffiths} and \cite{Renardy} for an explicit resolution of such boundary conditions in 1D.
	
	Adapting the idea of \cite{Bathory2017} to our setting is \UUU
        not a \EEE trivial task. \UUU Indeed, we need to tackle the
        additional problems \EEE  of adding the time evolution and of the intrinsic nonlinearity of the Navier-Stokes equation (which cannot be easily treated as a~constraint). It turns out that both these issues can be solved by a~clever choice of the functional $I$, whose Euler-Lagrange equations coincide (in the bulk) with the unsteady Navier-Stokes system in the sense of a~certain limit.
	
	\subsection{The  WIDE  functional} \UUU As already remarked in
        the Introduction, \EEE the Euler-Lagrange equation \eqref{EL0}
        \UUU of the WIDE functional $I_\eps$ is nothing but an 
          elliptic-in-time regularization of the original
          Navier-Stokes system \EEE
        \eqref{NS0}. 
        \UUU The use of WIDE variational approach \EEE can be tracked
        back \UUU at least to Ilmanen \cite{Ilmanen}, who used it to tackle 
existence and partial
regularity of the  Brakke mean-curvature flow of
varifolds. An application to existence of periodic
solutions for gradient flows is given by {
  Hirano}  \cite{Hirano94}.  The variational nature of elliptic
  regularization is at the core of \cite[Problem~3, p.~487]{Evans98}
  of the classical textbook by {Evans}. Two examples of
relaxation related to micro-structure evolution have been provided
 in  \cite{Conti-Ortiz08} and the case of mean-curvature evolution
of Cartesian surfaces is in \cite{spst}. The analysis of the WIDE
approach for abstract gradient flows for $\lambda$-convex and nonconvex energies is
in \cite{ms3,akst5} in the Hilbertian case and in \cite{cras,metric_wed} in
the metric case.  {Melchionna} \cite{Melchionna} extended the
theory to classes of nonpotential perturbations and {B\"ogelein, Duzaar, \&
Marcellini}~\cite{Boegelein-et-al14} used this variational approach to
prove the existence of variational
solutions to the equation 
$
u_t -\nabla \cdot f(x,u,\nabla u) + \partial_u f(x,u,
\nabla u)=0
$
where the field $f$ is convex in
$(u, \nabla u)$.  

Doubly nonlinear parabolic evolution equations have been tackled by the
WIDE variational formalism as
well.  The first result in this direction is by  {Mielke \& Ortiz}
\cite{Mielke-Ortiz06}, where the case of rate-independent processes is
addressed. The corresponding time
discretization has been presented in the subsequent \cite{ms2} and  an
application to crack-front propagation in brittle materials  is in
\cite{Larsen-et-al09}. 
The rate-dependent case  has been analyzed in 
\cite{akmest,akst,akst2,akst4}. See also   \cite{Liero-Melchionna} for
a stability result via $\Gamma$-convergence \cite{DalMaso93} and
  \cite{Melchionnas} for an application to the study of
symmetries of solutions.

In the dynamic case, {De Giorgi} conjectured in
\cite{DeGiorgi1996} that the WIDE functional procedure could be
implemented in the setting of semilinear waves.  This 
  has been ascertained in \cite{dg} (for the finite-time case)
and by  {Serra \& Tilli} \cite{Serra-Tilli10} (for the
infinite-time case).   The possibility of following this same
  variational approach in other hyperbolic situations has also been
  pointed out in \cite{DeGiorgi1996}. Indeed, extensions to mixed
hyperbolic-parabolic semilinear equations \cite{dg3}, to different
classes of nonlinear energies \cite{dg2,Serra-Tilli14}, and to
nonhomogeneous equations \cite{TTilli17, TTilli18} are also available.   
The validity of the WIDE approach to the wave equation in
time-dependent domain \cite{DalMaso-DeLuca}, dynamic perfect
plasticity \cite{Davoli16}. Eventually, the incompressible
Navier-Stokes system has been tackled in \cite{OSU}. \EEE 
	
	We build on the work \cite{OSU}, where it is shown that the
        Navier-Stokes equations can be obtained as limit of Euler-Lagrange equations of \UUU mimimizers \EEE of certain WIDE functionals,
        constrained to $\ve=0$ on the boundary. Here, we remove this
        constraint and, in effect, we leave to the functional to
        decide which boundary condition is optimal. Moreover, we also
        allow the non-Newtonian effects both in the bulk and on the
        boundary $\Gn$. This leads to the functional $I_{\eps}$
        defined in \eqref{I0}. The positive number $\eps$  tends to zero and it parametrizes the weight $e^{-\f{t}{\eps}}$, which dampens the impact of the future evolution on a~current state of the given system. The function $\fe$ represents an external body force density, such as gravity. Further, the function $\Sb:\Sym\to\Sym$ represents the stress-strain relation and the function $\Sb_{\eps}$ is given by
	\begin{equation}\label{Sepsdef}
		\Sb_{\eps}(A)\coloneqq\Sb(A)+\eps\sigma_4|A|^2A
		+\eps\sigma_q|A|^{q-2}A,\quad\sigma_4,\sigma_q>0,\quad q>1.
	\end{equation}
	The corresponding integral in the definition of $I_{\eps}$ is the amount of the dissipated energy due to internal friction. We postpone the formulation of precise assumptions on $\Sb$ to the next section, but for the time being, the reader may think of, e.g., the Ladyzhenskaya model
	\begin{equation}\label{lady1}
		\Sb(A)\coloneqq2(\sigma_2+\sigma_r^{\f2{r-2}}|A|^2)^{\f{r-2}2}A,\quad\sigma_2,\sigma_r>0,\quad r>2,
	\end{equation}
	satisfying obviously
	\begin{equation*}
		\Sb(A)\cdot A\geq\sigma_2|A|^2+\sigma_r|A|^r,
	\end{equation*}
	which is actually a~canonical representative of the class of models that is considered later. In the special case $r=2$, it is assumed that $\Sb$ is linear, i.e., $\Sb(A)=2\sigma_2 A$ corresponding to the classical Navier-Stokes model. The nonlinear model \eqref{lady1} is plausible also if $1<r<2$, leading to shear thinning fluids, but this case is excluded in the analysis below. The presence of $\eps>0$ in $\Sb_{\eps}$ further improves its asymptotic growth. This helps us in several places to deal with the convective term in the equation. The parameters $\sigma_2,\sigma_r$, $\sigma_4$, and $\sigma_q$ can be interpreted as certain generalized viscosities with units $[m^2s^{-1}]$, $[m^2s^{r-3}]$, $[m^2]$, and $[m^2s^{q-4}]$, respectively. The integrals with respect to $\lambda$ in the definition of $I_{\eps}$ are just potentials for the functions $\Sb_{\eps}$ and $\s_{\eps}$, see \eqref{Spot} below. These are used to model the viscous forces within the fluid and the friction forces on the wall $\Gn$. The mathematical role of the functions $\s$ and $\s_{\eps}$ is completely analogous to that of $\Sb$ and $\Sb_{\eps}$, respectively, and we shall impose
	\begin{equation}\label{sepsdef}
		\s_{\eps}(\uu)\coloneqq\s(\uu)+\eps\rho_4|\uu|^2\uu+\eps\rho_q|\uu|^{q-2}\uu,\quad\rho_4,\rho_q>0,\quad q>1,
	\end{equation}
	where, for the present time, we choose
	\begin{equation}\label{lady2}
		\s(\uu)\coloneqq2(\rho_2+\rho_r^{\f2{r-2}}|\uu|^2)^{\f{r-2}2}\uu,\quad\rho_2,\rho_r>0,\quad r>2,
	\end{equation}
	
	We do not claim that $I_{\eps}$ is easy to justify physically. At least some of its terms however bear a~clear physical meaning. First of all, the term $-\fe\cdot\ve$ obviously corresponds to the work done by external body forces. Secondly, note that
	\begin{equation*}
		\ii\int_0^1\Sb_{\eps}(\lambda\Dv)\cdot\Dv\dd{\lambda}+\inn\int_0^1\s_{\eps}(\lambda\ve)\cdot\ve\dd{\lambda}=\ii\T\cdot\Dv-\inn\T\n\cdot\ve
	\end{equation*}
	is the amount of dissipated energy in $\overline{\Omega}$ in some time instant by a~non-Newtonian fluid obeying the constitutive relation $\T=-\pr\I+\int_0^1\Sb_{\eps}(\lambda\Dv)\dd{\lambda}$ in $\Omega$ and the boundary condition $(\T\n)_{\tau}=-\int_0^1\s_{\eps}(\lambda\ve)_{\tau}\dd{\lambda}$ on $\Gn$ of the Navier-slip type, where $\T$ is the Cauchy stress tensor and $\pr$ is the pressure. Finally, the term $\eps|\doc{\ve}|^2$ represents the magnitude of certain inertial forces scaled by the parameter $\eps>0$. 
	
	We remark that $\eps$ has units of time and it models \UUU a
        \EEE certain future-time horizon that is still taken into account to get information about the present state. The causality is thus recovered only in the limit $\eps\to0_+$; we refer to \cite{OSU} for a~detailed discussion. The parameter $\eps$ has also a~secondary role as it introduces additional dissipation via $\Sb_{\eps}$ and $\s_{\eps}$, thus stabilizing the functional $I_{\eps}$. Similarly as in \cite{OSU}, we choose to work on the infinite time interval $(0,\infty)$, which avoids prescription of the terminal condition for velocity. This would be otherwise necessary since the Euler-Lagrange equation corresponding to $I_{\eps}$ is of the second order in time. To summarize, compared to the WIDE functional considered in \cite{OSU}, we make four important changes:
	
	1) We remove the constraint $\ve=0$ on $\partial\Omega$ and include friction on $\Gn$.
	
	2) We modify the inertial term by considering $\doc{\ve}$ instead of $\dot{\ve}\coloneqq\partial_t\ve+(\nn\ve)\ve$, in order to have the property
	\begin{equation}\label{canc}
		\doc{\ve}\cdot\ve=\partial_t(\tfrac12|\ve|^2).
	\end{equation}
	This identity is unaffected by the boundary conditions and it is crucial for obtaining an energy estimate for the solution of Euler-Lagrange equation of $I_{\eps}$. It seems that our method works for the standard form of the inertial term $|\dot{\ve}|^2$ only if $r>3$, leading to simpler outflow boundary conditions, without the corrector $\f12|\ve|^2\n$.
	
	3) We develop the whole theory for certain non-Newtonian fluids with nonlinear dependence of the stress on strain, both in the bulk and on the boundary.
	
	4) We use a~different stabilization term, which admits a~certain physical interpretation. Note that the aim of stabilizing $I_{\eps}$ is that of possibly obtaining information on $|\partial_t\ve|^2$ and $|\rot\ve\times\ve|^2$, starting from bounds on $|\doc{\ve}|^2$. This can be achieved by many different choices. It is an open question whether one can obtain similar results without any~stabilization of the WIDE functional, cf.\ \cite{OSU}.
	
	\subsection{Interpretation of \eqref{ac1} and \eqref{ac2}}
	
	Let us return to the boundary conditions \UUU  \eqref{ac1} and
        \eqref{ac2} from the \EEE minimization of $I_{\eps}$.
	
	Relation \eqref{ac1} just prescribes a~condition of the Navier-slip type on the impermeable part of the boundary $\Gn$. The standard Navier-slip corresponds to the case where  $r=2$ and $\s$ is linear.
	
	It is interesting to note that the first identity in \eqref{ac2} is strikingly similar to the stationary Navier-Stokes equation
	\begin{equation*}
		-\nn\pr-\di(\ve\otimes\ve)+\di\Sb(\Dv)=0.
	\end{equation*}
	This is actually quite intuitive since the outlet is just an
        abstract boundary, where {\it nothing should happen to the
          flow}, and therefore the equation that holds there should be
        just a {\it restriction} of the equation that is satisfied by the flow in the bulk. Moreover, since \eqref{ac2} is obviously nonlinear, it cannot be easily treated as a~constraint, unlike the usual do-nothing boundary conditions. This suggests that the outflow boundary conditions should be perceived as a~special kind of PDEs on the outlet boundary. Then, since we work with weak solutions, it is not at all surprising that we are able to identify the outflow boundary conditions only in a~weak sense. This has been also observed in \cite{BoyerFabrie} for a~slightly different type of outflow boundary condition. The analogous remark actually applies also to the tangential part of the Navier-slip-type boundary condition \eqref{ac1}.
	
	Without the corrector $\f12|\ve|^2\n$, relation \eqref{ac2} is sometimes called \textit{constant traction} boundary condition (cf.~\cite{Lanz}), which is just the do-nothing boundary condition \eqref{don}, but for the symmetric part of the velocity gradient. As we shall see, we have the freedom to replace $\Dv$ in the definition of $I_{\eps}$ by other types of gradients, leading, for example, to boundary conditions involving $\nn\ve$ rather than $\Dv$. However, this is at expense of losing the physical meaning of the dissipation terms in $I_{\eps}$ because of the possible failure of material frame indifference.
	
	It has been observed experimentally in \cite{Turek} that the
        correction $\f12|\ve|^2\n$ on the outflow boundary has
        a~positive effect on the flow characteristics. So far it has
        been unclear whether the boundary conditions such as
        \eqref{ac2} can be \UUU somehow \EEE derived, or if they are completely artificial, see the discussion in \cite{Neustupa2018}. Our result suggests, that there actually might be a~certain physical justification behind. On the other hand, in \cite{RannacherNS} it is argued that the outflow boundary condition
		\begin{equation}\label{turbc}
			-\pr\n-\tfrac12|\ve|^2\n+\nu(\nn\ve)\n=0\quad\text{on }\Gf^0,
		\end{equation}
		which is just a~version of \eqref{ac2} with the full velocity gradient (and renormalized pressure constant) is unphysical since it does not allow the Poiseuille flow in a~straight cylindrical pipe (the streamlines are bent inwards near the outlet, see the figures in \cite[Fig.~4~c)]{RannacherNS} or \cite{Turek1994}). This is certainly true if the outlet is flat, however, one has the freedom of prescribing \eqref{ac2} on a~curved outlet boundary (which is just an artificial interface, after all). This can partially compensate for the additional term $\f12|\ve|^2\n$, but not completely. In fact, choosing a~suitable shape of the outlet boundaries should be probably seen as a~part of the whole optimization problem. For simplicity however, in this work we assume that the outflow boundaries are fixed a~apriori.
	
	\subsection{On the pressure and \UUU the \EEE constants $c_i$}
	
	Let us provide more insight into the redefinition of the pressure, that was made from \eqref{gr0} to \eqref{gr}. It is well known that, in case of no inflows/outflows, i.e., if $\ve\cdot\n=0$ on $\partial\Omega$, the pressure can be completely eliminated from the Navier-Stokes equations by the Leray projection. This happens because $\ii\nn\pr\cdot\fit=0$ whenever $\di\fit=0$ and $\fit\cdot\n=0$ on $\partial\Omega$. Consequently, any substitution of the type $\pr(t,x)=\tilde{\pr}(t,x,\ve(t,x))$ does not change the problem. This is no longer true if $\ve\cdot\n\neq0$, in general. Indeed, whenever the boundary condition involves pressure (which we know it will, cf.~\eqref{ac2}), that pressure must correspond to the pressure appearing under the gradient operator in the Navier-Stokes equation.
	
	It is a~well known fact in the modeling of internal flows of incompressible fluids that the pressure is determined only up to a~constant (in space). This is immediately seen from \eqref{NS}. We wish to point out that, although the outflow boundary condition \eqref{ac2} involves pressure, it is actually invariant with respect to the shift of $\pr$ by a~constant due to \eqref{cs}, and therefore the whole system \eqref{NS}, \eqref{ac0}--\eqref{ac2} retains the same property. This is in agreement with the physical intuition that varying the pressure by a~same amount at all points of an incompressible fluid does not affect the flow, and there seems to be no reason why the presence of outflows should change this fact, since one can reasonably assume that the same fluid occupies the space also behind the outlet. Thus, one has always the freedom to impose a~single additional condition on $\pr$, such as the value at a~point, an integral average over some subdomain etc., provided that these quantities can be defined.
	
	There is an interesting analogy between the pressure $\pr$ and the constants $c_i$. While $\pr$ is a~Lagrange multiplier to the constraint $\di\ve=0$, the constants $c_i$ are Lagrange multipliers corresponding to $\int_{\Gf^i}\ve\cdot\n=F_i$, $i=0,\ldots,n$. Physically, constants $c_i$ represent certain generalized pressure drops (cf.\ \cite{Turek}) and they can all be shifted by a~common constant that is incorporated in the pressure $\pr$, without affecting the velocity $\ve$. It also seems possible to treat both $\pr$ and $c_i$ as unknowns of the system and compute them implicitly by minimizing the functional
	\begin{equation*}
		J_{\eps}(\ve,\pr,c_i)\coloneqq I_{\eps}(\ve)+\int_0^{\infty}e^{-\f t{\eps}}\Bigg(-\ii\pr\di\ve+\sum_{i=0}^nc_i\Big(\int_{\Gf^i}\ve\cdot\n-F_i\Big)\Bigg)+\mathcal{S}
	\end{equation*}
	over an enlarged function space without the constraints $\di\we=0$ and $\int_{\Gf^i}\we\cdot\n=0$. The implicit methods based on $J_{\eps}$ may turn out relevant in numerical implementations of the problem, since they give enhanced numerical stability and easier construction of the function spaces for the solution and test functions. However, this obviously leads to a~saddle point problem and some further stabilization $\mathcal{S}$ is necessary to ensure even the existence of a~minimum of $J_{\eps}$. Since it is hard to think of any physical justification behind $\mathcal{S}$ and there are many possible choices, we shall stick to the functional $I_{\eps}$, search for the solution in the spaces constrained by $\di\ve=0$, $\int_{\Gf^i}\ve\cdot\n=F_i$, and then construct $\pr$ and $c_i$ a~posteriori from $\ve$.

	In the remaining part of the paper, we provide a~rigorous counterpart of the procedure outlined in the introduction.
	
	\section{Technical assumptions \& definitions}\label{sec:tech}
	
	In this section, we state precisely the hypotheses needed to prove our main results. The definition of required function spaces is given here as well.
	
	\subsection{Constitutive assumptions for $\Sb$ and $\s$}
	
	Relations \eqref{lady1} and \eqref{lady2} \UUU are specified
        via \EEE assumptions on $\Sb$ and $\s$, which then allow us to apply our results to a~wide class of non-Newtonian fluids.
	
	For the function $\Sb$, we suppose that
	\begin{align}
		\Sb&\in\Cl^1(\Sym;\Sym),\label{SC}\\
		\partial_{ij}\Sb_{kl}&=\partial_{kl}\Sb_{ij}\quad\text{for all }i,j,k,l=1,2,3,\label{Spot}\\
		0&\leq(\Sb(A)-\Sb(B))\cdot(A-B),\label{Smon}\\
		|\Sb(A)|&\leq C(|A|+|A|^{r-1}),\label{Sup}\\
		\Sb(A)\cdot A&\geq\sigma_2|A|^2+\sigma_r|A|^r, \quad\sigma_2,\sigma_r>0,\label{Slow}
	\end{align}
	for all $A,B\in\Sym$ and some $r>1$. It is easy to see that under these conditions, the function $\Sb_{\eps}$, defined in \eqref{Sepsdef}, fulfills
	\begin{align}
		\Sb_{\eps}&\in\Cl^1(\Sym;\Sym),\label{SepsC}\\
		\partial_{ij}(\Sb_{\eps})_{kl}&=\partial_{kl}(\Sb_{\eps})_{ij}\quad\text{for all }i,j,k,l=1,2,3,\label{epsSpot}\\
		0&\leq(\Sb_{\eps}(A)-\Sb_{\eps}(B))\cdot(A-B),\label{epsSmon}\\
		|\Sb_{\eps}(A)|&\leq C(|A|+|A|^{r-1}+\eps|A|^3+\eps|A|^{q-1}),\label{epsSup}\\
		\Sb_{\eps}(A)\cdot A&\geq\sigma_2|A|^2+\sigma_r|A|^r+\eps\sigma_4|A|^4+\eps\sigma_q|A|^q, \quad\sigma_4,\sigma_q>0\label{epsSlow}
	\end{align}
	for all $A,B\in\Sym$ and some $q>1$. Note that thanks to \eqref{Spot}, there holds
	\begin{align}
		\f{\partial}{\partial A}\int_0^1\Sb_{\eps}(\lambda A)\cdot A\dd{\lambda}&\overset{\eqref{Spot}}{=}\int_0^1\Big(\lambda\f{\partial\Sb_{\eps}(\lambda A)}{\partial A}A+\Sb_{\eps}(\lambda A)\Big)\dd{\lambda}\nonumber\\
		&=\int_0^1\f{\dd{}}{\dd{\lambda}}(\lambda\Sb_{\eps}(\lambda A))\dd{\lambda}=\Sb_{\eps}(A).\label{pot}
	\end{align}
	Similarly, we require that $\s$ satisfies the properties:
	\begin{align}
		\s&\in\Cl^1(\R^3;\R^3),\label{sC}\\
		\partial_i\s_j&=\partial_j\s_i\quad\text{for all }i,j=1,2,3,\\
		0&\leq(\s(\uu)-\s(\we))\cdot(\uu-\we),\\
		|\s(\uu)|&\leq C(|\uu|+|\uu|^{r-1}),\\
		\s(\uu)\cdot\uu&\geq\rho_2|\uu|^2+\rho_r|\uu|^r,\quad\rho_2,\rho_r>0,\label{slow}
	\end{align}
	and, consequently, the function $\s_{\eps}$ defined in \eqref{sepsdef} fulfills
	\begin{align}
		\s_{\eps}&\in\Cl^1(\R^3;\R^3),\label{sepsC}\\
		\partial_i(\s_{\eps})_j&=\partial_j(\s_{\eps})_i\quad\text{for all }i,j=1,2,3,\label{sepspot}\\
		0&\leq(\s_{\eps}(\uu)-\s_{\eps}(\we))\cdot(\uu-\we),\\
		|\s_{\eps}(\uu)|&\leq C(|\uu|+|\uu|^{r-1}+\eps|\uu|^3+\eps|\uu|^{q-1}),\label{epssup}\\
		\s_{\eps}(\uu)\cdot\uu&\geq\rho_2|\uu|^2+\rho_r|\uu|^r+\eps\rho_4|\uu|^4+\eps\rho_q|\uu|^q,\quad\rho_4,\rho_q>0\label{sepslast}
	\end{align}
	and there holds
	\begin{equation}\label{pott}
		\f{\partial}{\partial\uu}\int_0^1\s_{\eps}(\lambda\uu)\cdot\uu\dd{\lambda}=\s_{\eps}(\uu)
	\end{equation}
	for all $\uu\in\R^3$. The assumptions \eqref{epsSpot} and \eqref{sepspot} could be omitted, however then, the identities \eqref{pot} and \eqref{pott}, characterizing the dissipation potentials, become more complicated, see \cite[Corollary.]{Edelen}. The above assumptions could be further generalized in many ways (e.g., by allowing anisotropic constitutive relations). We do not aim at maximum generality here and stick with the simple setting. The parameters $r$ and $q$ retain the same meaning throughout the whole paper and they will be eventually required to satisfy certain bounds. For a~future use, let us collect the above assumptions into the hypothesis ($\rm H_S$):
	\begin{equation}\label{HS}\tag{$\rm H_S$}
		\text{The functions }\Sb,\Sb_{\eps},\s,\s_{\eps}\text{ satisfy }\eqref{SC}\text{--}\eqref{Slow},\eqref{Sepsdef}\text{ and }\eqref{sC}\text{--}\eqref{slow},\eqref{sepsdef}.
	\end{equation}
	
	\subsection{The precise description of the domain $\Omega$}
	
	The domain $\Omega$ is an open bounded set in $\R^3$ (or $\R^2)$. Moreover, the domain $\Omega$ is assumed to be of class $\Cl^{0,1}$, i.e., Lipschitz (see \cite[Sect.~5.5.6]{John} or \cite[p.~49]{Necas} for definition), and nothing more. This allows us to consider domains with very sharp or obtuse corners, which may arise, for instance, if the outlets are cut under very sharp angles. This makes our results widely applicable, even to very rough domains. Of course, this also leads to very poor information about the pressure in the Navier-Stokes equations (since one cannot even apply the $L^p$-regularity for elliptic systems). Since our main result is of qualitative nature (identification of boundary conditions), this lack of information has limited effect in the analysis below.
	
	Next, we observe that, if the flux rates through the boundaries $\Gd$, $\Gf^1$ and $\Gf^2$ are prescribed, one can use the incompressibility constraint $\di\ve=0$ in order to infer that also the net flux rate through $\Gf^0$ is determined. Indeed we have the formula
	\begin{equation*}
		\int_{\Gf^0}\ve\cdot\n=-\int_{\Gd}\ve\cdot\n-\sum_{i=1}^n\int_{\Gf^i}\ve\cdot\n.
	\end{equation*}
	Hence, we assume throughout that the sets $\Omega$, $\Gd$, $\Gn$, $\Gf^i$, $i=0,\ldots,n$, and $\Gf$ are chosen in such a~way that
	\begin{equation}\left.
		\begin{aligned}
			&\Omega\subset\R^3\quad\text{is a~Lipschitz domain,}\\
			&\text{the sets}\quad\Gd,\Gn,\Gf\subset\partial\Omega\quad\text{are open and disjoint,}\qquad\\
			&\overline{\Gd\cup\Gn\cup\Gf}=\partial\Omega,\\
			&\text{each of the sets}\quad\Gf^i\quad\text{is connected,}\\
			&\text{the sets}\quad\overline{\Gf^i},\quad\text{are disjoint.}
		\end{aligned}\right\}\label{Hgeo}\tag{$\rm H_{\Omega}$}
	\end{equation}
	Moreover, we let
	\begin{equation*}
		\Gf\coloneqq\bigcup_{i=0}^n\Gf^i.
	\end{equation*}
	Consequently, the connected components of the set $\Gf$ are separated by (subsets of) $\Gd\cup\Gn$. In particular, if $|\Gd|=|\Gn|=0$, then $\partial\Omega=\Gf=\Gf^0$. Note that one could, in principle, allow existence of some neighboring outlets, but this would be physically counter-intuitive and present an unnecessary complication in the following.
	
	\subsection{Function spaces}	
	
	We \UUU classically \EEE denote the Lebesgue and Sobolev spaces by $(L^p(\Omega;\R^d),\norm{\cdot}_p)$ and $(W^{1,p}(\Omega;\R^d),\norm{\cdot}_{1,p})$, $1\leq p\leq\infty$, $d\in\N$, respectively. Next, let us define
	\begin{align*}
		\Cl^{\infty}_{\partial}&\coloneqq\left\{\uu\in\Cl^{\infty}(\overline{\Omega};\R^3):\uu\vert_{\Gd}=0,\;\uu\vert_{\Gn}\cdot\n=0,\;\int_{\Gf^i}\uu\cdot\n=0,\;i=0,\ldots,n\right\},\\
		\Cl^{\infty}_{\partial,\di}&\coloneqq\{\uu\in\Cl^{\infty}_{\partial}:\di\uu=0\}
	\end{align*}
	and then, for any $1<p<\infty$, we define $p'\coloneqq\f{p}{p-1}$ and
	\begin{align*}
		V^{1,p}&\coloneqq\overline{\Cl^{\infty}_{\partial}}{}^{\norm{\cdot}_{W^{1,p}(\Omega;\R^3)}},& V^{1,p}_{\di}&\coloneqq\overline{\Cl^{\infty}_{\partial,\di}}{}^{\norm{\cdot}_{W^{1,p}(\Omega;\R^3)}},\\
		V^{-1,p'}&\coloneqq (V^{1,p})',& V^{-1,p'}_{\di}&\coloneqq(V^{1,p}_{\di})'.
	\end{align*}
	The spaces $V^{1,p}$ and $V^{1,p}_{\di}$ are equipped with the norm
	\begin{equation*}
		\norm{\we}_{1,p}\coloneqq\norm{\we}_{W^{1,p}(\Omega;\R^3)}=\Big(\ii(|\nn\we|^p+|\we|^p)\Big)^{\f1p}.
	\end{equation*}
	Note that any $\we\in V^{1,p}$ satisfies $\int_{\partial\Omega}\we\cdot\n=0$ (and its boundary values can be thus extended to a~divergence-free vector field). In what follows, we often rely on the Korn-Poincar\'e inequality on $V^{1,p}$ in the form
	\begin{equation}\label{poin}
		\ii(|\nn\we|^p+|\we|^p)\leq c_p\Big(\ii|\D\we|^p+\inn|\we|^p\Big)\quad\text{for all }\we\in V^{1,p},
	\end{equation}
	which in our situation holds if $|\Gd|+|\Gn|>0$, a~physically reasonable assumption. Indeed, if $|\Gd|>0$, then it is a~standard result that \eqref{poin} holds even without the boundary term on the right-hand side. Further, if $|\Gd|=0$ and $|\Gn|>0$, then \eqref{poin} follows by a~slight adaptation of the argument from \cite[Lemma~1.11]{Indiana} (replacing the space $L^2(\partial\Omega)^d$ with $L^p(\Gn;\R^3)$ therein). Inequality \eqref{poin} cannot hold in the singular case $|\Gd|=|\Gn|=0$, where $\partial\Omega=\Gf^0$ and then it is obvious that unbounded constant vector fields violate \eqref{poin}. Note that due to \eqref{poin} and the trace theorem, the expression $\norm{\D\we}_p+\norm{\we}_{p;\Gn}$ is an equivalent norm on $V^{1,p}$.
	
	We shall also need the following special spaces of Lions-Magenes type defined on a~an open connected subset $G$ of $\partial\Omega$. First, let us denote
	\begin{align*}
		W^{1,\infty}_0(G;\R^3)&\coloneqq\{\we\in W^{1,\infty}(G;\R^3):\we=0\text{ on }\partial G\},\\
		W^{1,\infty}_{0,\n}(G;\R^3)&\coloneqq\{\we\in W^{1,\infty}_0(G;\R^3):\we\cdot\n=0\text{ on }G\}
	\end{align*}
	and then, we put
	\begin{align}
		W^{\f1{p'},p}_{0}(G;\R^3)&\coloneqq\overline{W^{1,\infty}_0(G;\R^3)^{\norm{\cdot}_{W^{1/p',p}(G;\R^3)}}},\label{LM0}\\ W^{\f1{p'},p}_{0,\n}(G;\R^3)&\coloneqq\overline{W^{1,\infty}_{0,\n}(G;\R^3)^{\norm{\cdot}_{W^{1/p',p}(G;\R^3)}}}.\label{LM0n}
	\end{align}
	The boundary conditions will be identified in the corresponding dual spaces.
	
	Next, we need to introduce some spaces for time dependent functions. We set $Q_T\coloneqq(0,T)\times\Omega$ for all $T>0$, including $T=\infty$. The Bochner spaces are denoted as $(L^p(0,T;X);\norm{\cdot}_{L^p(0,T;X)})$, where $X$ is a~Banach space. If $X=L^s(\Omega;\R^d)$, or $X=W^{1,s}(\Omega;\R^d)$, $1\leq s\leq \infty$, we make an abbreviation $\norm{\cdot}_{L^p(0,T;L^s)}\coloneqq\norm{\cdot}_{L^p(0,T;L^s(\Omega;\R^d))}$, or $\norm{\cdot}_{L^p(0,T;W^{1,s})}\coloneqq\norm{\cdot}_{L^p(0,T;W^{1,s}(\Omega;\R^d))}$, respectively. Further, we define
	\begin{align*}
		\X^p\coloneqq L^p(0,\infty;V^{1,p})\quad\text{and}\quad
		\X^p_{\di}\coloneqq L^p(0,\infty;V^{1,p}_{\di}).
	\end{align*}
	Next, we define the space of admissible trajectories starting from the zero initial datum as
	\begin{equation*}
		\Ul_0\coloneqq\{\uu\in L^q_{\rm loc}(0,\infty;V^{1,q}_{\di}):\partial_t\uu\in L^2_{\rm loc}(0,\infty;L^2(\Omega;\R^3)),\;\uu(0)=0\}
	\end{equation*}
	and, we denote the corresponding spaces of test functions as
	\begin{align*}
		\Vl_0&\coloneqq\{\fit\in\Ul_0:I_1(\fit)<\infty\},\\
		\Vl_c&\coloneqq\{\pit\in \Vl_0:t\mapsto\pit(t,\cdot)\text{ has a~compact support in }(0,\infty)\}.
	\end{align*}	
	
	\subsection{Nonhomogeneous data}
	
	In the problem under consideration, we wish to prescribe a~nonhomogeneous initial datum $\uu_0:\Omega\to\R^3$ and boundary data $\ve_{\rm D}:(0,\infty)\times\Gd\to\R^3$, $F_i:(0,\infty)\to\R$, $i=0,\ldots,n$. However, due to the nonlinearity of the problem, it seems difficult to formulate some explicit necessary conditions on this data that are needed for the existence of the corresponding global-in-time weak solution of Navier-Stokes equations, especially if the data are allowed to be time-dependent. We refer to \cite{Amann}, \cite{Fursikov}, \cite[IX.4]{galdi}, or \cite{Raymond} where this topic is (partially) treated using different approaches. Basically, one is asking whether it is possible to extend the boundary data $\ve_{\rm D}$, $F_i$ to a~divergence free field $\ve_0$ of such regularity that the products $(\nn\ve)\ve_0$ and $(\nn\ve_0)\ve$ are under control. Obviously, this depends on many factors and there seems to be no agreement on how this extension should be constructed. To avoid this difficult question and, at the same time, to make our results applicable in real scenarios with (large) inflow and outflows, we assume that the data $\uu_0$, $\ve_{\rm D}$, $F_i$ and $\fe$ are admissible in the sense that for any $\eps>0$ there exists a~function $\ve_0^{\eps}:(0,\infty)\times\Omega\to\R^3$ with the following properties:
	\begin{equation}\label{H0A}\tag{$\rm H_0^A$}\left.\begin{aligned}
			&\di\ve^{\eps}_0=0,\quad\ve^{\eps}_0\vert_{\Gd}=\ve_{\rm D},\quad\ve^{\eps}_0\vert_{\Gn}\cdot\n=0,\quad\int_{\Gf^i}\ve^{\eps}_0\cdot\n=F_i,\\
			&\lim_{\eps\to0_+}\norm{\ve_0^{\eps}(0)-\uu_0}_2\to0,\\
			&\norm{\ve_0^{\eps}}_{\X^2\cap \X^r}+\norm{\eps^{\f14}\ve_0^{\eps}}_{\X^4}+\norm{\eps^{\f1q}\ve_0^{\eps}}_{\X^q}+\norm{\eps^{\f12}\partial_t\ve_0^{\eps}}_{L^2(0,\infty;L^2)}\leq C
		\end{aligned}\right\}\end{equation}
	and eventually also
	\begin{equation}\label{H0B}\tag{$\rm H_0^B$}
		\norm{\partial_t\ve_0^{\eps}}_{(\X^2_{\di}\cap \X^r_{\di})'}+\norm{\ve_0^{\eps}}_{L^{\infty}(0,\infty;L^2)\cap L^{\f r{r-2}}(0,\infty;L^{\f r{r-2}})}\leq C.
	\end{equation}
	
	Admissible data indeed exist, at least in some simple scenarios. An obvious case is when $\ve_{\rm D}$ and $F_i$ are time independent and $\uu_0\in V^{1,q}_{\di}$. Then, we can simply choose $\ve_0^{\eps}(t)\coloneqq\uu_0$ for all $t>0$ and $\eps>0$.
	
	It is easy to see that if $\ve_0^{\eps}$ exist, then there also exists a~function $\ve_0\in \X^2_{\di}\cap \X^r_{\di}\cap L^{\infty}(0,\infty;L^2(\Omega;\R^3))$ with $\partial_t\ve_0\in(\X^2_{\di}\cap \X^r_{\di})'$ and satisfying the same constraints as $\ve_0^{\eps}$, to which a~subsequence of $\{\ve_{\eps}\}_{\eps>0}$ converges weakly in the corresponding spaces (and there also holds $\ve(0)=\uu_0$, $\di\uu_0=0$).
	
	For the external body force density, let us assume (for simplicity) that
	\begin{equation}\label{Hf}\tag{$\rm H_{\fe}$}
		\fe\in L^2(0,\infty;L^2(\Omega;\R^3)).
	\end{equation}

	\section{\UUU Main result: variational resolution of
          Navier-Stokes equations with outflow \EEE boundary
          conditions}\label{sec:Tex}

        \UUU We are now in the position of stating the main result of
        the paper, which makes the oulined variational approach
        rigorous. \EEE
	
	\begin{theorem}\label{Tex}
		Let the hypotheses \eqref{Hgeo}, \eqref{HS}, \eqref{H0A} and \eqref{Hf} be fulfilled. Then, for every $\eps>0$, the functional $I_{\eps}$ attains a~minimum $\ve_{\eps}$ in the set $\Ul^{\eps}\coloneqq\Ul_0+\ve_0^{\eps}$ and the function $\ve_{\eps}$ satisfies
		\begin{align}\label{EL2}
			\int_0^{\infty}\ii\doc{\ve}_{\eps}&\cdot(\pit+\eps\partial_t\pit+\eps\rot\ve_{\eps}\times\pit+\eps\rot\pit\times\ve_{\eps})\nonumber\\
			&+\int_0^{\infty}\Big(\ii\Sb_{\eps}(\Dv_{\eps})\cdot\D\pit+\inn\s_{\eps}(\ve_{\eps})\cdot\pit\Big)=\int_0^{\infty}\ii\fe\cdot\pit
		\end{align}
		for all $\pit\in\Vl_c$.
		
		If, in addition, the conditions \eqref{H0B} and
		\begin{equation}\label{bou}
			\f{11}5\leq r<4,\qquad 4<q\leq3r',\qquad\min(\sigma_4,\rho_4)>\f{c_4}4,
		\end{equation}
		hold, then there exists a~function $\ve$ with the properties
		\begin{align}
			\ve-\ve_0&\in \X^2_{\di}\cap \X^r_{\di}\cap L^{\infty}(0,\infty;L^2(\Omega;\R^3)),\label{NSpro1}\\
			\partial_t\ve&\in(\X^2_{\di}\cap \X^r_{\di})',\label{NSpro2}\\
			\ve(0)&=\uu_0,\label{NSpro3}
		\end{align}
		which is a~limit of a~not relabeled subsequence of $\{\ve_{\eps}\}_{\eps>0}$ in the sense that
		\begin{align}
			\ve_{\eps}&\wc\ve&&\text{weakly in }\X^2\cap\X^r,\label{cc1}\\
			\ve_{\eps}&\wc\ve&&\text{weakly}\ast\text{ in }L^{\infty}(0,\infty;L^2(\Omega;\R^3)),\label{cc2}\\
			\ve_{\eps}&\to\ve&&\text{a.e.\ in }(0,\infty)\times\Omega\text{ and on }(0,\infty)\times\Gn,\label{cc4}\\
			\partial_t\ve_{\eps}&\wc\partial_t\ve&&\text{weakly in }(\X^2_{\di}\cap \X^q_{\di})',\label{cc7}
		\end{align}
		and which solves
		\begin{equation}\label{NSdiv}
			\int_0^{\infty}\!\!\Big(\scal{\partial_t\ve,\fit}+\int_{\Omega}(\rot\ve\times\ve)\cdot\fit+\ii\Sb(\Dv)\cdot\D\fit+\inn\s(\ve)\cdot\fit\Big)=\int_0^{\infty}\!\!\ii\fe\cdot\fit
		\end{equation}
		for all $\fit\in \X^2_{\di}\cap \X^r_{\di}$.
		
		Furthermore, let $D\in L^{\infty}_{\rm loc}(0,\infty;\R)$. Then, there exists a~function
		\begin{equation}\label{Ppr}
			P\in L^{\infty}_{\rm loc}(0,\infty;L^{r'}(\Omega;\R))\quad\text{with}\quad\ii P=D
		\end{equation}
		and such that
		\begin{align}
			&\int_0^{\infty}\Big(-\ii\ve\cdot\partial_t\pit+\ii(\nn\ve)\ve\cdot\pit+\ii\Sb(\Dv)\cdot\D\pit+\inn\!\!\!\s(\ve)\cdot\pit\Big)\nonumber\\
			&\quad+\int_0^{\infty}\Big(\ii P\di\partial_t\pit-\f12\int_{\Gf}\!\!\!|\ve|^2\n\cdot\pit\Big)
			=\int_0^{\infty}\!\!\ii\fe\cdot\pit.\label{P}
		\end{align}
		for all $\pit\in\Cl^{\infty}_c((0,\infty);V^{1,r})$.
		
		Finally, choose any $\etb_i\in W^{1,\infty}_0(\Gf^i;\R^3)$ satisfying $\int_{\Gf^i}\etb_i\cdot\n=1$, $i=0,\ldots,n$ and, for any $\psi\in\Cl^1_c((0,\infty);\R)$, let us define
		\begin{align}
			\T_{\psi}&\coloneqq\int_0^{\infty}(P\I\partial_t\psi+\Sb(\Dv)\psi)\nonumber\\
			(c_i)_{\psi}&\coloneqq\scal{\T(\psi)\n,\etb_i}_{\Gf^i}-\f12\int_{\Gf^i}\int_0^{\infty}|\ve|^2\n\cdot\etb_i\psi.\label{cs2}
		\end{align}
		Then, the functions $\ve$ and $P$ satisfy the boundary conditions
		\begin{align}
			\ve\vert_{\Gd}&=\ve_{\rm D},&&&\label{bc1}\\
			\ve\vert_{\Gn}\cdot\n&=0,&\T_{\psi}\n+\int_0^{\infty}\s(\ve)\psi&=0&\text{in }(W^{\f1{r'},r}_{0,\n}(\Gn;\R^3))',\label{bc3}\\
			\int_{\Gf^i}\ve\cdot\n&=F^i,&\T_{\psi}\n-\f12\int_0^{\infty}|\ve|^2\n\psi&=(c_i)_{\psi}\n&\text{in }(W_0^{\f1{r'},r}(\Gf^i;\R^3))',\label{bc2}
		\end{align}
		for every $i=0,\ldots,n$ and for all $\psi\in\Cl^1_c((0,\infty);\R)$.
	\end{theorem}
\UUU Theorem~\ref{Tex} is proved in Section \ref{sec:proof} below. We
devote the rest of this section to several comments instead:
\EEE

	(i) Since the setting of Theorem~\ref{Tex} enables to test \eqref{NSdiv} by $\ve-\ve_0$, it is clear that we are dealing with a weak solution of the Leray-Hopf type. The total energy of the fluid can however be increased by the non-homogeneous and possibly unsteady inflow, therefore the standard form of the energy (in)equality is unavailable. This can be also seen in the proof of the uniform estimate below, where the physical energy of the (approximated) solution is estimated by the size of the data.
		
	(ii) If, for any reason, the solution $(\ve,P)$ turns out to be sufficiently smooth up to the boundary $\Gn\cup\Gf$, then the boundary conditions \eqref{bc1}--\eqref{bc2}, \eqref{cs2} simplify to \eqref{ac0}--\eqref{ac2}, \eqref{cs}, i.e., they hold in the classical sense.
	
	(iii) Note that if the test function $\fit$ from \eqref{NSdiv} vanishes on the whole $\partial\Omega$, then
	\begin{equation}
		\ii(\rot\ve\times\ve)\cdot\fit=\ii(\nn\ve)\ve\cdot\fit-\f12\ii\di(|\ve|^2\fit)=\ii(\nn\ve)\ve\cdot\fit,\label{gra}
	\end{equation}
	which is the standard form of the convective term in the Navier-Stokes equations. If $\fit$ does not vanish on $\partial\Omega$, the term $-(\nn\ve)^T\ve\cdot\fit$ contributes both to a~pressure and to a~boundary term as can be seen in \eqref{P}.
	
	(iv) In assumption \eqref{bou}, the parameters $q$, $\sigma_4$, and $\rho_4$ correspond just to the stabilization of $I_{\eps}$ and their values have no impact on the properties of $(\ve,P)$.
	
	(v) The upper bounds $r<4$ and $q\leq 3r'$ are easily removable. The inequality $q\leq 3r'$ just simplifies some dual spaces that are needed below. The case $r\geq4$ is sub-critical and the analysis becomes simpler as the stabilization of $I_{\eps}$ by the fourth order dissipation terms is unnecessary, i.e., one can set $\sigma_4=\rho_4=0$. Furthermore, if $r>4$, we put $\sigma_4=\rho_4=\sigma_q=\rho_q=0$ and the parameter $q$ is omitted completely.
	
	(vi) The lower bound $r\geq\f{11}5$ is probably not optimal. However, in the case $r\in(\f65,\f{11}5)\setminus\{2\}$, the identification of the weak limit of $\Sb(\Dv_{\eps})$ becomes difficult. It is not obvious whether the standard methods (such as the $L^{\infty}$- or $W^{1,\infty}$- truncations) can be applied directly for our $\eps$-approximation scheme, which is rather complicated. Also, in the case $r<2$, there is another difficulty of extending the boundary data without some kind of smallness condition on $\ve_{\rm D}\cdot\n$ and $F_i$, cf.\ e.g.\ \cite{Stara} and \cite{Lanz2009}.
	
	(vii) Theorem~\ref{Tex} is valid, of course, for the classical Navier-Stokes system, where $r=2$ and $\Sb,\s$ are linear functions. However, the spaces for $\partial_t\ve$ and for $\fit$ in \eqref{NSdiv} have to be modified in a~standard way (it is no longer possible to consider $\ve-\ve_0$ as a~test function in \eqref{NSdiv}). Also, it is apparent in the proof below (cf.~\eqref{rotp}) that if $r=2$, one has to replace the assumption $\norm{\ve_0^{\eps}}_{L^{\infty}(0,\infty;L^{\infty})}\leq C$ (recall \eqref{H0B}) by the requirement that the number
	\begin{equation*}
		\esssup_{t>0}\sup_{\substack{0\neq\fit\in W^{1,2}(\Omega;\R^3)\\\di\fit=0}}\ii\ve_0^{\eps}(t)\cdot\f{\rot\fit\times\fit}{\norm{\fit}_{1,2}^2}
	\end{equation*}	
	can be made smaller than any given $\delta>0$. That $\ve_0^{\eps}$ can indeed be constructed in such a~way follows from \cite[(IX.4.43)]{galdi} and it is closely related to the so-called \textit{extension condition} of Leray and Hopf.
	
	(viii) With usual modifications, the above theorem can be proved also in the $d$-dimensional, $d\geq2$, setting.
	
	(iv) Starting with the definition of $I_{\eps}$, it is possible to replace everywhere the symmetric velocity gradient $\Dv$ by the full velocity gradient $\nn\ve$. This changes very little in the analysis below. Then, the analogue of Theorem~\ref{Tex} yields the outflow boundary condition
	\begin{equation*}
		-\pr\n-\tfrac12|\ve|^2\n+\Sb(\nn\ve)\n=c_i\n\quad\text{on }\Gf^i
	\end{equation*}
	(again, if the solution is sufficiently smooth). This is a~nonlinear version of the do-nothing boundary condition with the dynamic pressure correction. One could even replace $\Dv$ by $\nn\ve+a(\nn\ve)^T$ for any $a\neq-1$ and again, since $\di(\nn\ve)^T=\nn\di\ve=0$, this would not change anything except the resulting boundary conditions.
	
	(x) If $r>3$, then $(\nn\ve)\ve\cdot\ve$ becomes easily controllable and hence, we do not need to use \eqref{canc}. Thus, by replacing $\doc{\ve}$ with $\dot{\ve}$ in the definition of $I_{\eps}$, it is possible to prove a~version of Theorem~\ref{Tex}, where the convective terms take the usual form $(\nn\ve)\ve$ and the condition \eqref{bc2} becomes just (a weak version of)
	\begin{equation*}
		-\pr\n-\Sb(\nn\ve)\n=c_i\n\quad\text{on }\Gf^i.
	\end{equation*}
	By choosing $\Sb(A)=\nu A+\delta|A|^{\delta+1}A$, $\delta>0$, and making a~convenient shift in the pressure, this yields 
	\begin{equation}\label{don2}
		-\pr\n+\nu(\nn\ve)\n+\delta|\nn\ve|^{\delta+1}(\nn\ve)\n=0\quad\text{on }\Gf^0.
	\end{equation}
	Making $\delta>0$ very small, this is as close as we can get to the classical do-nothing boundary condition \eqref{don}, while retaining the global-in-time and large-data existence of a~weak solution of the Navier-Stokes equations corresponding to $\Sb$.
	
	(xi) By subtracting the term $\int_0^{\infty}e^{-\f t{\eps}}\int_{\Gf}\gb\cdot\ve$ in the definition of $I_{\eps}$, we can include also a~forcing through the outflow boundaries, that may arise as a~reaction force to the fluid flow outside $\Omega$. Then, if $\gb\in L^2(0,\infty;L^2(\Gf;\R^3))$, this additional term can be handled as the one with $\fe$, thanks to the fact that we are able to control the whole Sobolev norms of the solution via \eqref{poin}. (In fact, one could be slightly more general here and observe that $\fe,\gb\in L^2L^2+L^rL^r$ is sufficient to obtain Theorem~\ref{Tex}.) This way, it is possible to derive a~nonhomogeneous version of the boundary condition \eqref{bc2}. On the other hand, it seems hard to imagine any kind of physical device that would yield exactly some given forcing $\gb$, apart from the case where $\gb$ arises as a~boundary value of some potential, in which case it can be included in the pressure. Thus, to ensure a~viable physical interpretation, one should probably stick with the choice $\gb=0$, which is what we do in this work.
	
	(xii) Adding the term $\int_0^{\infty}e^{-\f t{\eps}}\int_{\Gn}\f{\eps}2|\partial_t\ve|^2$ in the definition of $I_{\eps}$ leads (at least formally) to the so-called \textit{dynamic slip} boundary conditions, which are studied in \cite{HATZIKIRIAKOS} for certain polymer melts. Applying the similar idea also on the outflow boundary $\Gf$, one can recover the conditions involving $\partial_t\ve$, that are studied analytically in \cite{Bothe} and which are quite popular in numerics, see, e.g., \cite{Li} and references therein. For simplicity of presentation (it would require major changes in the definitions of the function spaces below), we do not discuss these extensions in detail here.
	
	(xiii) The distribution $\pr\coloneqq\partial_t P$ from the
        above theorem is called \textit{pressure}. Note that the
        boundary conditions are identified only in an averaged sense
        over time. This is to avoid the fact that $\di\T$, where
        $\T\coloneqq-\pr\I+\Sb(\Dv)$, is not a~function, in general,
        and hence the trace of $\T\n$ cannot be defined
        consistently. Actually, this issue arises already for the
        non-stationary Stokes system with, e.g., the Navier-slip
        boundary condition if the domain or the initial data are rough
        and no additional regularity of solution, besides the energy
        estimate, is available. On the other hand, one can read from
        \eqref{P} that $\di\T_{\psi}$ is an integrable function,
        hereby providing a~possibility to define the object
        $\T_{\psi}\n$ in a~way which is compatible with the common
        meaning of a~trace (a limit of restrictions of smooth
        functions) and with the corresponding integration-by-parts
        formula, that is essential for the identification of
        \eqref{bc3} and \eqref{bc2}. For a~different resolution of
        this issue, see \cite{BoyerFabrie}. It is clear that if
        $\T,\di\T\in L^1(Q_{\infty};\R^3)$, then $(\cdot)_{\psi}$
        averaging is unnecessary and conditions \eqref{bc3},
        \eqref{bc2} hold a.e.\ in $(0,\infty)$. The functions $\etb_i$
        always exist (see the end of the proof of Theorem~\ref{Tex}
        below) and their only purpose is to get around the fact that
        the functional $\T_{\psi}\n$ cannot be applied to $\n$, for
        a~general Lipschitz domain.

        \section{\UUU Proof of the main result}\label{sec:proof}

        \UUU The \EEE proof of Theorem~\ref{Tex} \UUU is developed
        throughout the section and and is divided in subsequent
        sections and structured \EEE on a~series of lemmas. In what follows, the number $\omega\in(0,\f12)$ is systematically used as an auxiliary parameter (arising, e.g., from the Young inequality), which is eventually chosen sufficiently small.
	
	\subsection{Existence of minima}
	
	Let us consider the problem of minimizing $I_{\eps}$ over the space $\Ul^{\eps}$. The following lemma tells us that this problem is solvable.
	
	\begin{lemma}\label{Lex}
		Suppose that \eqref{Hgeo}, \eqref{HS}, \eqref{H0A}, and \eqref{Hf} are satisfied. Then, the functional $I_{\eps}$ attains a~minimum $\ve_{\eps}$ in $\Ul^{\eps}$. Moreover, there exists a~constant $C>0$, depending only on the data and $\Omega$, such that
		\begin{equation}\label{mins}
			|I_{\eps}(\ve_{\eps})|\leq C\quad\text{for all}\quad\eps>0.
		\end{equation}
	\end{lemma}
	\begin{proof}
		Let us start by deriving some preliminary estimates that are used below for estimation of certain terms of $I_{\eps}$.
		
		Let $\we\in\Ul^{\eps}$. We use \eqref{poin} to estimate
		\begin{equation*}
			\ii|\rot\we\times\we|^2\leq\ii|\rot\we|^2|\we|^2\leq 2\ii|\nn\we|^2|\we|^2\leq\ii(|\nn\we|^4+|\we|^4),
		\end{equation*}
		thus
		\begin{equation}\label{rotw0}
			\norm{\rot\we\times\we}_2\leq\norm{\we}_{1,4}^2.
		\end{equation}
		As a~consequence of the convexity of the power function $z\mapsto z^{\alpha}$, $1<\alpha<\infty$, there holds
		\begin{equation}\label{youngus}
			(x+y)^{\alpha}\leq(1+\omega)x^{\alpha}+\f{1+\omega}{((1+\omega)^{\alpha'-1}-1)^{\alpha-1}}y^\alpha\quad\text{for all }x,y\geq0.
		\end{equation}
		Hence, using also \eqref{poin}, we can continue with the estimate \eqref{rotw0} to get
		\begin{align}
			\ii|\rot\we&\times\we|^2\leq(1+\omega)\norm{\we-\ve_0^{\eps}}_{1,4}^4+C(\omega)\norm{\ve_0^{\eps}}_{1,4}^4\label{rotw}\nonumber\\
			&\leq(1+\omega)c_4\Big(\ii|\D(\we-\ve_0^{\eps})|^4+\inn|\we-\ve_0^{\eps}|^4\Big)+C(\omega)\norm{\ve_0^{\eps}}_{1,4}^4\nonumber\\
			&\leq(1+\omega)^2c_4\Big(\ii|\D\we|^4+\inn|\we|^4\Big)+C(\omega)\norm{\ve_0^{\eps}}_{1,4}^4.
		\end{align}
		From this, Young's inequality and $\omega<\f12$, we also deduce that
		\begin{align}
			\ii|\doc{\we}|^2&\geq\ii(\omega|\partial_t\we|^2-\f{\omega}{1-\omega}|\rot\we\times\we|^2)\nonumber\\
			&\geq\omega\ii|\partial_t\we|^2-5\omega c_4\ii|\D\we|^4-5\omega c_4\inn|\we|^4-C(\omega)\norm{\ve_0^{\eps}}_{1,4}^4\label{docest}
		\end{align}
		Next, we estimate the term with $\fe$ using Young's inequality, \eqref{youngus}, and \eqref{poin} as
		\begin{align}
			\Big|&\ii\fe\cdot\we\Big|\leq\omega\norm{\we}_2^2+C(\omega)\norm{\fe}_2^2\nonumber\\
			&\leq\omega(1+\omega)\norm{\we-\ve_0^{\eps}}^2_2+C(\omega)(\norm{\ve_0^{\eps}}_2^2+\norm{\fe}_2^2)\nonumber\\
			&\leq\omega(1+\omega)c_2\Big(\ii|\D(\we-\ve_0^{\eps})|^2+\inn|\we-\ve_0^{\eps}|^2\Big)+C(\omega)(\norm{\ve_0^{\eps}}_2^2+\norm{\fe}_2^2)\nonumber\\
			&\leq\omega(1+\omega)^2c_2\Big(\ii|\D\we|^2+\inn|\we|^2\Big)+C(\omega)(\norm{\ve_0^{\eps}}_{1,2}^2+\norm{\fe}_2^2).\label{fest}
		\end{align}
		Further, by \eqref{epsSup}, we have
		\begin{align}
			\int_0^1\Sb_{\eps}(\lambda\D\we)\cdot\D\we\dd{\lambda}&\leq C\int_0^1\lambda|\D\we|^2+\lambda^{r-1}|\D\we|^r+\eps\lambda^3|\D\we|^4+\eps\lambda^{q-1}|\D\we|^q\dd{\lambda}\nonumber\\
			&\leq C(|\D\we|^2+|\D\we|^r+\eps|\D\we|^4+\eps|\D\we|^q)\label{Slaup}
		\end{align}
		and analogously, by \eqref{epssup}, also
		\begin{equation}\label{slaup}
			\int_0^1\s_{\eps}(\lambda\we)\cdot\we\dd{\lambda}\leq C(|\we|^2+|\we|^r+\eps|\we|^4+\eps|\we|^q)
		\end{equation}
		On the other hand, we use \eqref{epsSmon} and \eqref{epsSlow} to obtain the following estimate from below:
		\begin{align}
			\int_0^1\Sb_{\eps}(\lambda\D\we)\cdot\D\we\dd{\lambda}&\overset{\eqref{epsSlow}}{\geq}\int_{\f12}^1\Sb_{\eps}(\lambda\D\we)\cdot\D\we\dd{\lambda}\nonumber\\
			&\overset{\eqref{epsSmon}}{\geq}\Sb_{\eps}(\tfrac12\D\we)\cdot(\tfrac12\D\we)\nonumber\\
			&\overset{\eqref{epsSlow}}{\geq}2^{-2}\sigma_2|\D\we|^2+2^{-r}\sigma_r|\D\we|^r\nonumber\\
			&\qquad+\eps2^{-4}\sigma_4|\D\we|^4+\eps2^{-q}\sigma_q|\D\we|^q.\label{Seps}
		\end{align}
		Analogously, we also get
		\begin{equation}\label{Seps2}
			\int_0^1\s_{\eps}(\lambda\we)\cdot\we\dd{\lambda}\geq2^{-2}\rho_2|\we|^2+2^{-r}\rho_r|\we|^r+\eps2^{-4}\rho_4|\we|^4+\eps2^{-q}\rho_q|\we|^q.
		\end{equation}
		
		Now we proceed with the proof of existence of a~minimum. Since $0\in\Ul_0$, we have $\ve_0^{\eps}\in \Ul^{\eps}$ and the set $\Ul^{\eps}$ is thus nonempty. As $\we\in \Ul^{\eps}$ implies $\partial_t\we\in L^2_{\rm loc}(0,\infty;L^2(\Omega;\R^3))$, $\we\in L^r_{\rm loc}(0,\infty;W^{1,r}(\Omega;\R^3))$, and $\we\in L^r_{\rm loc}(0,\infty;L^r(\Gn;\R^3))$, $1\leq r\leq q$, we see, using \eqref{rotw0}, \eqref{Slaup}, and \eqref{slaup} that the integral
		\begin{align*}
			I_{\eps}^T(\we)&\coloneqq\int_0^Te^{-\f{t}{\eps}}\Big(\f{\eps}2\ii|\doc{\we}|^2-\ii\fe\cdot\we\nonumber\\
			&\qquad+\ii\int_0^1\Sb_{\eps}(\lambda\D\we)\cdot\D\we\dd{\lambda}+\inn\int_0^1\s_{\eps}(\lambda\we)\cdot\we\dd{\lambda}\Big)
		\end{align*}
		is well defined and finite for every $T>0$. Moreover, estimates \eqref{rotw0}, \eqref{Slaup}, and \eqref{slaup} together with assumption \eqref{H0A} imply
		\begin{align}
			I_{\eps}^T(\ve_0^{\eps})&\leq C\int_0^T\Big(\ii\big(\eps|\partial_t\ve_0^{\eps}|^2+\eps|\ve_0^{\eps}|^4+\eps|\nn\ve_0^{\eps}|^4+|\fe|^2+|\ve_0^{\eps}|^2\big)\nonumber\\
			&\qquad+\ii\big(|\Dv_0^{\eps}|^2+|\Dv_0^{\eps}|^r+\eps|\Dv_0^{\eps}|^4+\eps|\Dv_0^{\eps}|^q\big)\nonumber\\
			&\qquad+\inn\big(|\ve_0^{\eps}|^2+|\ve_0^{\eps}|^r+\eps|\ve_0^{\eps}|^4+\eps|\ve_0^{\eps}|^q\big)\Big)\leq C,\label{upest}
		\end{align}
		for all $T>0$, hence $\inf_{\Ul^{\eps}}I_{\eps}\leq I_{\eps}(\ve_0^{\eps})<\infty$. On the other hand, using Young's inequality, \eqref{fest}, and \eqref{H0A}, we get
		\begin{align}\label{Icoer}
			I_{\eps}^T(\we)&\geq C\int_0^Te^{-\frac{t}{\eps}}\Big(\ii|\D\we|^2+\inn|\we|^2-\ii\fe\cdot\we\big)\Big)\geq-C
		\end{align}
		for all $T>0$, therefore $I_{\eps}(\we)\geq -C$ for every $\we\in\Ul^{\eps}$. This and \eqref{upest} proves \eqref{mins} once we show that a~minimum is attained.
		
		Let $\{\ve_k\}_{k=1}^\infty\subset \Ul^{\eps}$ be a~minimizing sequence for $I_{\eps}$, satisfying also $I_{\eps}(\ve_k)<C$ for all $k\in\N$. By using \eqref{docest}, \eqref{fest}, \eqref{Seps}, \eqref{Seps2} and then choosing $\omega>0$ sufficiently small, we get
		\begin{align*}
			I_{\eps}(\ve_k)&\geq C\int_0^{\infty}e^{-\f{t}{\eps}}\Big(\eps\omega\ii|\partial_t\ve_k|^2-\eps\omega\ii|\D\ve_k|^4-\eps\omega\inn|\ve_k|^4-\eps C(\omega)\norm{\ve_0^{\eps}}_{1,4}^4\big)\\
			&\quad-\omega\ii|\D\ve_k|^2-\omega\inn|\ve_k|^2-C(\omega)(\norm{\ve_0^{\eps}}_{1,2}^2+\norm{\fe}_2^2)\\
			&\quad+\ii(|\D\ve_k|^2+|\D\ve_k|^r+\eps|\Dv_k|^4+\eps|\D\ve_k|^q)\\
			&\quad+\inn(|\ve_k|^2+|\ve_k|^r+\eps|\ve_k|^4+\eps|\ve_k|^q)\Big)\\
			&\geq C\Big(\eps\norm{e^{-\f{t}{2\eps}}\partial_t\ve_k}_{L^2(0,\infty;L^2)}^2+\norm{e^{-\f{t}{2\eps}}\ve_k}_{L^2(0,\infty;V^{1,2})}^2+\norm{e^{-\f{t}{r\eps}}\ve_k}_{L^r(0,\infty;V^{1,r})}^r\\
			&\quad+\eps\norm{e^{-\f{t}{4\eps}}\ve_k}_{L^4(0,\infty;V^{1,4})}^4+\eps\norm{e^{-\f{t}{q\eps}}\ve_k}_{L^q(0,\infty;V^{1,q})}^q\Big)-C.
		\end{align*}
		Using this estimate on $(0,T)$, $T>0$, we can apply the Aubin-Lions lemma with the compact embedding $W^{1,2}(\Omega;\R^3)\embl L^2(\Omega;\R^3)$ to deduce that a~nonrelabeled subsequence of $\ve_k$ converges to $\ve_{\eps}\in \Ul^{\eps}$ strongly in $L^2(0,T;L^2(\Omega;\R^3))$ and $\nn\ve_k$ converges weakly to $\nn\ve_{\eps}$. Therefore, the product $\rot\ve_k\times\ve_k$, which is bounded in $L^2(0,T;L^2(\Omega;\R^3))$ (recall \eqref{rotw}), converges weakly (up to another subsequence) to the correct limit $\rot\ve_{\eps}\times\ve_{\eps}$. Hence, by the weak lower semi-continuity of norms, we obtain
		\begin{align*}
			I_{\eps}^T(\ve_{\eps})+\int_0^Te^{-\f{t}{\eps}}\ii\fe\cdot\ve_{\eps}&\leq\liminf_{k\to\infty}\Big(I_{\eps}(\ve_k)+\int_0^{\infty}e^{-\f t{\eps}}\ii\fe\cdot\ve_k\Big)\\
			&=\inf_{\Ul^{\eps}}I_{\eps}+\int_0^{\infty}e^{-\f{t}{\eps}}\ii\fe\cdot\ve_{\eps}
		\end{align*}
		and, taking the limit $T\to\infty$, we see that $\ve_{\eps}$ is a~minimum of $I_{\eps}$.
              \end{proof}

              \subsection{\UUU Euler-Lagrange equations}
	
	Next, we derive the Euler-Lagrange equation of $I_{\eps}$, which we state in three different ways: \eqref{EL1} is useful for deriving most of the $\eps$-uniform estimates, \eqref{EL2} is convenient for the weak limit identification, while \eqref{EL3} is important to get an $\eps$-uniform estimate of $\partial_t\ve_{\eps}$.
	
	\begin{lemma}\label{LEL}
		Assume that \eqref{Hgeo}, \eqref{HS}, \eqref{H0A}, and \eqref{Hf} hold and let $\ve_{\eps}$ be a~minimum of $I_{\eps}$ in $\Ul^{\eps}$. Then
		\begin{align}\label{EL1}
			\int_0^{\infty}e^{-\f{t}{\eps}}\Big(\ii\eps\doc{\ve}_{\eps}&\cdot(\partial_t\fit+\rot\ve_{\eps}\times\fit+\rot\fit\times\ve_{\eps})-\ii\fe\cdot\fit\nonumber\\
			&+\ii\Sb_{\eps}(\Dv_{\eps})\cdot\D\fit+\inn\s_{\eps}(\ve_{\eps})\cdot\fit\Big)=0
		\end{align}
		for all $\fit\in \Vl_0$. Moreover, we have \eqref{EL2} for all $\pit\in \Vl_c$. Furthermore, if
		\begin{equation}\label{ineq}
			2\leq r\leq 4\qquad\text{and}\qquad4\leq q\leq 3r',
		\end{equation}
		then the relation
		\begin{equation}\label{EL3}
			-\eps\partial_t\doc{\ve}_{\eps}+\doc{\ve}_{\eps}+\eps^{\f12-\f1q}N^1_{\eps}\ve_{\eps}+\eps^{1-\f3q}N^2_{\eps}\ve_{\eps}+\eps^{\f14}A^1_{\eps}\ve_{\eps}+\eps^{\f1q}A^2_{\eps}\ve_{\eps}+A\ve_{\eps}=\fe
		\end{equation}
		holds as an identity in the space $L^1_{\rm loc}(0,\infty;V^{-1,q'}_{\di})$, where we define
		\begin{align*}
			\scal{N^1_{\eps}\ve_{\eps},\fit}&\coloneqq\eps^{\f12+\f1q}\ii\partial_t\ve_{\eps}\cdot(\rot\ve_{\eps}\times\fit+\rot\fit\times\ve_{\eps}),\\
			\scal{N^2_{\eps}\ve_{\eps},\fit}&\coloneqq\eps^{\f3q}\ii(\rot\ve_{\eps}\times\ve_{\eps})\cdot(\rot\ve_{\eps}\times\fit+\rot\fit\times\ve_{\eps}),\\
			\scal{A^1_{\eps}\ve_{\eps},\fit}&\coloneqq\eps^{\f34}\sigma_4\ii|\Dv_{\eps}|^2\Dv_{\eps}\cdot\D\fit+\eps^{\f34}\rho_4\inn|\ve_{\eps}|^2\ve_{\eps}\cdot\fit,\\
			\scal{A^2_{\eps}\ve_{\eps},\fit}&\coloneqq\eps^{\f1{q'}}\sigma_q\ii|\Dv_{\eps}|^{q-2}\Dv_{\eps}\cdot\D\fit+\eps^{\f1{q'}}\rho_q\inn|\ve_{\eps}|^{q-2}\ve_{\eps}\cdot\fit,\\
			\scal{A\ve_{\eps},\fit}&\coloneqq\ii\Sb(\Dv_{\eps})\cdot\D\fit+\inn\s(\ve_{\eps})\cdot\fit,
		\end{align*}
		for all $\fit\in V^{1,q}$.
	\end{lemma}
	\begin{proof}
		Let $\fit\in \Vl_0$. Since $\Vl_0\subset\Ul_0$, we have, for all $s\in\R$, that $\ve_{\eps}+s\fit\in \Ul^{\eps}$ and it is easy to see that also $I(\ve_{\eps}+s\fit)<\infty$. Let us define
		\begin{equation*}
			g(s)\coloneqq\int_0^1\Sb_{\eps}(\lambda(\Dv_{\eps}+s\D\fit))\cdot(\Dv_{\eps}+s\D\fit)\dd{\lambda}.
		\end{equation*}
		By using the fundamental theorem of calculus in the form
		\begin{equation*}
			g(s)-g(0)=s\int_0^1g'(us)\dd{u}
		\end{equation*}
		together with \eqref{pot}, we obtain
		\begin{align*}
			g(s)-g(0)&=s\int_0^1\f{\partial}{\partial A}\int_0^1\Sb_{\eps}(\lambda A)\cdot A\dd{\lambda}\Big\vert_{A=\Dv_{\eps}+us\D\fit}\cdot\f{\dd{}}{\dd{\tau}}(\Dv_{\eps}+\tau\D\fit)\Big\vert_{\tau=us}\dd{u}\\
			&\overset{\eqref{pot}}{=}s\int_0^1\Sb_{\eps}(\Dv_{\eps}+us\D\fit)\cdot\D\fit\dd{u}.
		\end{align*}
		An analogous identity holds, of course, for the function $\s_{\eps}$. Hence, also using
		\begin{equation*}
			\f12(|\we_1|^2-|\we_2|^2)=\we_2\cdot(\we_1-\we_2)+\f12|\we_1-\we_2|^2,
		\end{equation*}
		we find that
		\begin{align}
			0&\leq I(\ve_{\eps}+s\fit)-I(\ve_{\eps})\nonumber\\
			&=\int_0^{\infty}e^{-\f{t}{\eps}}\Big(s\ii \eps\doc{\ve}_{\eps}\cdot(\partial_t\fit+\rot\ve_{\eps}\times\fit+\rot\fit\times\ve_{\eps}+s\rot\fit\times\fit)\nonumber\\
			&\quad+\f{s^2}2\ii|\partial_t\fit+\rot\ve_{\eps}\times\fit+\rot\fit\times\ve_{\eps}+s\rot\fit\times\fit|^2-s\ii \fe\cdot\fit\nonumber\\
			&\quad +s\ii\int_0^1\Sb_{\eps}(\Dv_{\eps}+us\D\fit)\cdot\D\fit\dd{u}+s\inn\int_0^1\s_{\eps}(\ve_{\eps}+us\fit)\cdot\fit\dd{u}\Big).\label{gato}
		\end{align}
		Dividing this by $s>0$, using $I_{\eps}(\fit)<\infty$ and the continuity of $\Sb_{\eps}$ and $\s_{\eps}$ to take the limit $s\to0_+$, and then doing the analogous procedure for $s<0$ leads to \eqref{EL1}.
		
		If $\pit\in \Vl_c$, then also $\fit\coloneqq e^{\f{t}{\eps}}\pit\in \Vl_c\subset \Vl_0$ and, since $\partial_t\fit=e^{\f{t}{\eps}}(\partial_t\pit+\eps^{-1}\pit)$, we obtain \eqref{EL2} from \eqref{EL1}.
		
		To prove \eqref{EL3}, note first using H\"older's inequality and the fact that $\ve_{\eps}\in\Ul^{\eps}$ that the functionals $N^1_{\eps}\ve_{\eps}$, $N^2_{\eps}\ve_{\eps}$, $A^1_{\eps}\ve_{\eps}$, $A^2_{\eps}\ve_{\eps}$, and $A\ve_{\eps}$ are well-defined on $V^{1,\f{2q}{q-2}}$, $V^{1,\f{q}{q-3}}$, $V^{1,4}$, $V^{1,q}$ and $V^{1,r}$, respectively, for almost every time,
		where the space $V^{1,q}$ is the smallest one of these. Indeed, this a~consequence of the inequality 
		\begin{equation*}
			2\leq r\leq\f{q}{q-3}\leq4\leq\f{2q}{q-2}\leq 3r',
		\end{equation*}
		that follows \eqref{ineq}. Hence, if we rewrite \eqref{EL2} as
		\begin{align*}
			&\int_0^{\infty}\ii\eps\doc{\ve}_{\eps}\cdot\partial_t\pit+\int_0^{\infty}\langle\doc{\ve}_{\eps}+\eps^{\f12-\f1q}N^1_{\eps}\ve_{\eps}+\eps^{1-\f3q}N^2_{\eps}\ve_{\eps}\\
			&\qquad+\eps^{\f14}A^1_{\eps}\ve_{\eps}+\eps^{\f1q}A^2_{\eps}\ve_{\eps}+A\ve_{\eps}-\fe,\pit\rangle=0,
		\end{align*}
		we deduce that $\partial_t\doc{\ve}_{\eps}\in L^1_{\rm loc}(0,\infty;V^{-1,q'}_{\di})$, leading to \eqref{EL3}.
              \end{proof}

              \subsection{\UUU A priori estimates}
	
	Now we examine the limit of the sequence of minimizers constructed in the previous section. First, we need to derive an $\eps$-uniform estimate for $\ve_{\eps}$, which is the key technical point of the paper.
	
	\begin{lemma}\label{Luni}
		Suppose that \eqref{Hgeo}, \eqref{HS}, \eqref{H0A}, \eqref{H0B}, and \eqref{Hf} hold and let $\{\ve_{\eps}\}_{\eps>0}$ be a~sequence of minimizers to $\{I_{\eps}\}_{\eps>0}$. If $r\geq\f{11}5$ and $\min(\sigma_4,\rho_4)>\f{c_4}4$, then
		\begin{align}
			\norm{\ve_{\eps}}_{\X^2\cap \X^r}+\norm{\eps^{\f14}\ve_{\eps}}_{\X^4}+\norm{\eps^{\f1q}\ve_{\eps}}_{\X^q}+\norm{\ve_{\eps}}_{L^{\infty}(0,\infty;L^2)}&\leq C,\label{odhad1}\\
			\norm{\eps^{\f12}\partial_t\ve_{\eps}}_{L^2(0,\infty;L^2)}+\norm{\eps^{\f12}\rot\ve_{\eps}\times\ve_{\eps}}_{L^2(0,\infty;L^2)}+\norm{\eps^{\f12}\doc{\ve}_{\eps}}_{L^2(0,\infty;L^2)}&\leq C,\label{odhad2}\\
			\norm{\rot\ve_{\eps}\times\ve_{\eps}}_{L^{\f r2}(0,\infty;L^{\f r2})}+\norm{\eps^{\f2q}\rot\ve_{\eps}\times\ve_{\eps}}_{L^{\f q2}(0,\infty;L^{\f q2})}&\leq C,\label{odhad3}\\
			\norm{\partial_t\ve_{\eps}}_{(\X^2_{\di}\cap \X^q_{\di})'}+\norm{\rot\ve_{\eps}\times\ve_{\eps}}_{(\X^r)'}+\norm{\doc{\ve}_{\eps}}_{(\X^2_{\di}\cap \X^q_{\di})'}&\leq C.\label{odhad4}
		\end{align}
	\end{lemma}
	\begin{proof}
		To get an estimate on $(0,\eps)$, we notice that $e^{-\f t{\eps}}\geq e^{-1}$ on that interval, hence we can just apply \eqref{mins}, use Young's inequality and \eqref{poin} to eliminate $\fe$ and then estimate the exponential from below. This leads to
		\begin{align}
			C\geq I_{\eps}(\ve_{\eps})&\geq C\int_0^{\infty}e^{-\f t{\eps}}\Big(\ii(\eps|\doc{\ve}_{\eps}|^2+|\Dv_{\eps}|^2+|\Dv_{\eps}|^r+\eps|\Dv_{\eps}|^4+\eps|\Dv_{\eps}|^q)\nonumber\\
			&\qquad+\inn(|\ve_{\eps}|^2+|\ve_{\eps}|^r+\eps|\ve_{\eps}|^4+\eps|\ve_{\eps}|^q)\Big)-C\nonumber\\
			&\qquad\geq C\int_0^{\eps}\Big(\ii(\eps|\doc{\ve}_{\eps}|^2+|\Dv_{\eps}|^2+|\Dv_{\eps}|^r+\eps|\Dv_{\eps}|^4+\eps|\Dv_{\eps}|^q)\nonumber\\
			&\qquad+\inn(|\ve_{\eps}|^2+|\ve_{\eps}|^r+\eps|\ve_{\eps}|^4+\eps|\ve_{\eps}|^q)\Big)-C.\label{malej}
		\end{align}
		
		Next, we let $T>0$ and choose $\fit\coloneqq\eta(\ve_{\eps}-\ve_0^{\eps})$ in \eqref{EL1}, where
		\begin{equation*}
			\eta(t)=\Big\{\begin{aligned}
				&e^{\f{t}{\eps}}-1,\quad t\leq T\\
				&e^{\f{T}{\eps}}-1,\quad t>T.
			\end{aligned}
		\end{equation*}
		Note that then $\fit\in \Vl_0$ and $\fit(0)=0$. Hence, we obtain
		\begin{align}\label{O}
			&\int_0^{\infty}e^{-\f{t}{\eps}}\eta\Big(\ii\eps\doc{\ve}_{\eps}\cdot(\partial_t\ve_{\eps}+2\rot\ve_{\eps}\times\ve_{\eps})+\ii\Sb_{\eps}(\Dv_{\eps})\cdot\Dv_{\eps}+\inn\s(\ve_{\eps})\cdot\ve_{\eps}\Big)\nonumber\\
			&\qquad+\int_0^{\infty}\eps e^{-\f{t}{\eps}}\eta'\ii\partial_t(\tfrac12|\ve_{\eps}|^2)\nonumber\\
			&\quad=\int_0^{\infty}e^{-\f{t}{\eps}}\eta\Big(\ii\eps\doc{\ve}_{\eps}\cdot(\partial_t\ve_0^{\eps}+\rot\ve_{\eps}\times\ve_0^{\eps}+\rot\ve_0^{\eps}\times\ve_{\eps})+\ii\fe\cdot(\ve_{\eps}-\ve_0^{\eps})\nonumber\\
			&\qquad+\ii\Sb_{\eps}(\Dv_{\eps})\cdot\Dv_0^{\eps}+\inn\s(\ve_{\eps})\cdot\ve_0^{\eps}\Big)+\int_0^{\infty}\eps e^{-\f{t}{\eps}}\eta'\ii\doc{\ve}_{\eps}\cdot\ve_0^{\eps},
		\end{align}
		where we used \eqref{canc} in order to eliminate the term $(\nn\ve_{\eps})\ve_{\eps}\cdot\ve_{\eps}$, which would otherwise be impossible to control for general boundary conditions and $r\leq3$. Next, we apply the identity
		\begin{equation}\label{a+b}
			\doc{\ve}_{\eps}\cdot(\partial_t\ve_{\eps}+2\rot\ve_{\eps}\times\ve_{\eps})=|\partial_t\ve_{\eps}+\tfrac32\rot\ve_{\eps}\times\ve_{\eps}|^2-\tfrac14|\rot\ve_{\eps}\times\ve_{\eps}|^2,
		\end{equation}
		we note that $\eps e^{-\frac{t}{\eps}}\eta'(t)=\chi_{(0,T)}(t)$ and $\ve_{\eps}(0)=\ve_0^{\eps}(0)$ to rewrite \eqref{O} as
		\begin{align}
			&\int_0^{\infty}e^{-\f{t}{\eps}}\eta\Big(\ii\eps|\partial_t\ve_{\eps}+\tfrac32\rot\ve_{\eps}\times\ve_{\eps}|^2+\ii\Sb_{\eps}(\Dv_{\eps})\cdot\Dv_{\eps}+\inn\s_{\eps}(\ve_{\eps})\cdot\ve_{\eps}\Big)\nonumber\\
			&\quad+\f12\ii|\ve_0^{\eps}(0)|^2+\f12\ii|\ve_{\eps}(T)|^2\nonumber\\
			&=\int_0^{\infty}e^{-\f{t}{\eps}}\eta\Big(\ii\eps(\partial_t\ve_{\eps}+\tfrac32\rot\ve_{\eps}\times\ve_{\eps})\cdot(\partial_t\ve_0^{\eps}+\rot\ve_{\eps}\times\ve_0^{\eps}+\rot\ve_0^{\eps}\times\ve_{\eps})\nonumber\\
			&\quad+\f14\ii\eps|\rot\ve_{\eps}\times\ve_{\eps}|^2-\frac12\ii\eps(\rot\ve_{\eps}\times\ve_{\eps})\cdot(\partial_t\ve_0^{\eps}+\rot\ve_{\eps}\times\ve_0^{\eps}+\rot\ve_0^{\eps}\times\ve_{\eps})\nonumber\\
			&\quad+\ii\Sb_{\eps}(\D\ve_{\eps})\cdot\D\ve_0^{\eps}+\inn\s_{\eps}(\ve_{\eps})\cdot\ve_0^{\eps}+\ii\fe\cdot(\ve_{\eps}-\ve_0^{\eps})\Big)\nonumber\\
			&\quad+\ii|\ve_0^{\eps}(0)|^2+\int_0^{T}\ii\doc{\ve}_{\eps}\cdot\ve_0^{\eps}.\label{tested}
		\end{align}
		Our aim is to estimate the terms on the right-hand side by the $\omega$-Young inequality, so that the part containing $\ve_{\eps}$ is sufficiently small and the other (large) part depends only on $\ve_0^{\eps}$, for which we can use assumptions \eqref{H0A}, \eqref{H0B}, \eqref{Hf} and include these terms into a~constant $C$ (or $C(\omega)$). Let us proceed term by term and observe first, using \eqref{poin}, that
		\begin{align}
			&\eps\ii|\partial_t\ve_0^{\eps}+\rot\ve_{\eps}\times\ve_0^{\eps}+\rot\ve_0^{\eps}\times\ve_{\eps}|^2\nonumber\\
			&\quad\leq 3\eps\ii(|\partial_t\ve_0^{\eps}|^2+|\rot\ve_{\eps}\times\ve_0^{\eps}|^2+|\rot\ve_0^{\eps}\times\ve_{\eps}|^2)\nonumber\\
			&\quad\leq 3\eps\ii(|\partial_t\ve_0^{\eps}|^2+2|\nn\ve_{\eps}|^2|\ve_0^{\eps}|^2+2|\nn\ve_0^{\eps}|^2|\ve_{\eps}|^2)\nonumber\\
			&\quad\leq\omega\eps C\Big(\ii|\D\ve_{\eps}|^4+\inn|\ve_{\eps}|^4\Big)+\eps C\norm{\partial_t\ve_0^{\eps}}_2^2+\eps C(\omega)\norm{\ve_0^{\eps}}_{1,4}^4.\label{sum3}
		\end{align}	
		The products $\eps|\rot\ve_{\eps}\times\ve_{\eps}|^2$ are already estimated in \eqref{rotw}. Note that the assumptions on $\sigma_4$ and $\rho_4$ are needed to absorb the term $\f14\eps|\rot\ve_{\eps}\times\ve_{\eps}|^4$ which is the only term on the right-hand side of \eqref{tested} that does not involve data. Next, for the term with $\Sb_{\eps}$, we use \eqref{Sup} and Young's inequality to get
		\begin{align}
			\ii\Sb_{\eps}(\D\ve_{\eps})\cdot\D\ve_0^{\eps}&\leq C\ii(|\Dv_{\eps}|+|\Dv_{\eps}|^{r-1}+\eps|\Dv_{\eps}|^3+\eps|\Dv_{\eps}|^{q-1})|\D\ve_0^{\eps}|\nonumber\\
			&\leq\omega C\ii(|\D\ve_{\eps}|^2+|\D\ve_{\eps}|^r+\eps|\D\ve_{\eps}|^4+\eps|\D\ve_{\eps}|^q)\nonumber\\
			&\qquad+C(\omega)\ii(|\Dv_0^{\eps}|^2+|\Dv_0^{\eps}|^r+\eps|\Dv_0^{\eps}|^4+\eps|\Dv_0^{\eps}|^q)\label{aux}
		\end{align}
		and, similarly, we also obtain
		\begin{align}
			\inn\s_{\eps}(\ve_{\eps})\cdot\ve_0^{\eps}&\leq \omega C\inn(|\ve_{\eps}|^2+|\ve_{\eps}|^r+\eps|\ve_{\eps}|^4+\eps|\ve_{\eps}|^q)\nonumber\\
			&\qquad+C(\omega)\inn(|\ve_0^{\eps}|^2+|\ve_0^{\eps}|^r+\eps|\ve_0^{\eps}|^4+\eps|\ve_0^{\eps}|^q).\label{aux2}
		\end{align}
		To handle the term containing $\fe$, we proceed analogously as in \eqref{fest}, leading to
		\begin{equation}\label{fest2}
			\Big|\ii\fe\cdot(\ve_{\eps}-\ve_0^{\eps})\Big|\leq\omega C\Big(\ii|\Dv_{\eps}|^2+\inn|\ve_{\eps}|^2\Big)+C(\omega)(\norm{\ve_0^{\eps}}_{1,2}^2+\norm{\fe}_2^2).
		\end{equation}
		Regarding the last two terms in \eqref{tested}, we use \eqref{poin} to get
		\begin{align}
			\ii(\rot\ve_{\eps}\times\ve_{\eps})\cdot\ve_0^{\eps}&\leq\f{2\omega}{r}\ii|\rot\ve_{\eps}|^{\f r2}|\ve_{\eps}|^{\f r2}+\f{r-2}{r\omega}\ii|\ve_0^{\eps}|^{\f{r}{r-2}}\nonumber\\
			&\leq C\omega\ii|\D\ve_{\eps}|^r+C\omega\inn|\ve_{\eps}|^r+C(\omega)\norm{\ve_0^{\eps}}_{\f{r}{r-2}}^{\f{r}{r-2}}\label{rotp}
		\end{align}
		and then the Bochner version of integration by parts formula to write
		\begin{align*}
			&\ii|\ve_0^{\eps}(0)|^2+\int_0^T\ii\doc{\ve}_{\eps}\cdot\ve_0^{\eps}\\
			&\qquad=\ii\ve_{\eps}(T)\cdot\ve_0^{\eps}(T)-\int_0^T\scal{\partial_t\ve_0^{\eps},\ve_{\eps}}+\int_0^T\ii(\rot\ve_{\eps}\times\ve_{\eps})\cdot\ve_0^{\eps}\\
			&\leq\f14\ii|\ve_{\eps}(T)|^2+\omega C\Big(\int_0^T\ii(|\Dv_{\eps}|^2+|\D\ve_{\eps}|^r)+\int_0^T\inn(|\ve_{\eps}|^2+|\ve_{\eps}|^r)\Big)\\
			&\qquad+C(\omega)\inf_{\we_1+\we_2=\partial_t\ve_0^{\eps}}(\norm{\we_1}_{(\X^2_{\di})'}^2+\norm{\we_2}_{(\X^r_{\di})'}^{r'})\\
			&\qquad+C(\omega)\norm{\ve_0^{\eps}}_{L^{\f{r}{r-2}}(0,\infty;L^{\f{r}{r-2}})}^{\f{r}{r-2}}+\norm{\ve_0^{\eps}}_{L^{\infty}(0,\infty;L^2)}^2.
		\end{align*}
		Using this, \eqref{sum3}, \eqref{aux}, and \eqref{aux2} in \eqref{tested}, choosing $\omega>0$ sufficiently small and recalling \eqref{H0A}, \eqref{H0B}, we get
		\begin{align}
			&\int_0^{\infty}e^{-\f{t}{\eps}}\eta\Big(\ii\big(\eps|\partial_t\ve_{\eps}+\tfrac32\rot\ve_{\eps}\times\ve_{\eps}|^2\big)\nonumber\\
			&\qquad+\ii(|\D\ve_{\eps}|^2+|\D\ve_{\eps}|^r+\eps|\Dv_{\eps}|^4+\eps|\Dv_{\eps}|^q)\nonumber\\
			&\qquad+\inn(|\ve_{\eps}|^2+|\ve_{\eps}|^r+\eps|\ve_{\eps}|^4+\eps|\ve_{\eps}|^q)\Big)+\f12\ii|\ve_0^{\eps}(0)|^2+\f14\ii|\ve_{\eps}(T)|^2\nonumber\\
			&\leq\omega C\Big(\int_0^T\ii(|\Dv_{\eps}|^2+|\D\ve_{\eps}|^r)+\int_0^T\inn(|\ve_{\eps}|^2+|\ve_{\eps}|^r)\Big)+\norm{\ve_0^{\eps}}_{L^{\infty}(0,\infty;L^2)}^2\nonumber\\
			&\qquad +C(\omega)\int_0^{\infty}\big(\norm{\ve_0^{\eps}}_{1,2}^2+\norm{\ve_0^{\eps}}_{1,r}^r+\eps\norm{\ve_0^{\eps}}_{1,4}^4+\eps\norm{\ve_0^{\eps}}_{1,q}^q+\eps\norm{\partial_t\ve_0^{\eps}}_2^2\big)\nonumber\\
			&\qquad+C(\omega)\inf_{\we_1+\we_2=\partial_t\ve_0^{\eps}}(\norm{\we_1}_{(\X^2_{\di})'}^2+\norm{\we_2}_{(\X^r_{\di})'}^{r'})+C(\omega)\norm{\ve_0^{\eps}}_{L^{\f{r}{r-2}}(0,\infty;L^{\f{r}{r-2}})}^{\f{r}{r-2}}\nonumber\\
			&\leq\omega C\Big(\int_{\eps}^T\ii(|\Dv_{\eps}|^2+|\D\ve_{\eps}|^r)+\int_{\eps}^T\inn(|\ve_{\eps}|^2+|\ve_{\eps}|^r)\Big)+C(\omega),\label{es}
		\end{align}
		where in the last inequality we also used \eqref{malej} to estimate the integral over $(0,\eps)$ by a~constant. If we apply the inequality
		\begin{equation*}
			e^{-\f t{\eps}}\eta\geq(1-e^{-\f{t}{\eps}})\chi_{(\eps,T)}\geq(1-e^{-1})\chi_{(\eps,T)},\quad t>0,
		\end{equation*}
		on the left-hand side of \eqref{es}, we see that the integral on the right-hand side of \eqref{es} gets absorbed for $\omega$ sufficiently small, leading to
		\begin{align*}
			&\int_{\eps}^T\Big(\ii\big(\eps|\partial_t\ve_{\eps}+\tfrac32\rot\ve_{\eps}\times\ve_{\eps}|^2\big)\nonumber\\
			&\qquad+\ii(|\D\ve_{\eps}|^2+|\D\ve_{\eps}|^r+\eps|\Dv_{\eps}|^4+\eps|\Dv_{\eps}|^q)\nonumber\\
			&\qquad+\inn(|\ve_{\eps}|^2+|\ve_{\eps}|^r+\eps|\ve_{\eps}|^4+\eps|\ve_{\eps}|^q)\Big)+\ii|\ve_{\eps}(T)|^2\leq C.
		\end{align*}
		Putting this information together with \eqref{malej} and then taking the essential supremum over $T>0$, we arrive at
		\begin{align}
			&\int_0^{\infty}\Big(\ii(|\D\ve_{\eps}|^2+|\D\ve_{\eps}|^r+\eps|\Dv_{\eps}|^4+\eps|\Dv_{\eps}|^q)\nonumber\\
			&\qquad+\inn(|\ve_{\eps}|^2+|\ve_{\eps}|^r+\eps|\ve_{\eps}|^4+\eps|\ve_{\eps}|^q)\Big)+\esssup_{(0,\infty)}\ii|\ve_{\eps}|^2\leq C\label{dis}
		\end{align}
		and also at
		\begin{align}\label{anis}
			&\int_0^{\eps}\ii\eps|\doc{\ve}_{\eps}|^2+\int_{\eps}^{\infty}\ii\eps|\doc{\ve}_{\eps}+\tfrac12\rot\ve_{\eps}\times\ve_{\eps}|^2\leq C.
		\end{align}
		By virtue of \eqref{poin}, estimate \eqref{dis} is equivalent to \eqref{odhad1}. Moreover, the information $\norm{\eps^{\f14}\ve_{\eps}}_{\X^4}\leq C$ implies that $\norm{\eps^{\f12}\rot\ve_{\eps}\times\ve_{\eps}}_2\leq C$ via \eqref{rotw0} which, together with \eqref{anis}, yields \eqref{odhad2} through the Young inequality. Moreover, by a~similar estimate to \eqref{rotp}, we also immediately obtain \eqref{odhad3} from \eqref{odhad1}.
		
		To extract information about $\partial_t\ve_{\eps}$, we need first to estimate $\rot\ve_{\eps}\times\ve_{\eps}$ in an appropriate dual space. In the case $r\geq3$, we use the information that $\rot\ve_{\eps}$ is bounded in $L^2(0,\infty;L^2(\Omega;\R^3))$ and that $\ve_{\eps}$ is bounded in $L^{\infty}(0,\infty;L^2(\Omega;\R^3))\cap L^2(0,\infty;L^6(\Omega;\R^3))$ (using the Sobolev embedding), which leads to
		\begin{equation}\label{rotpr}
			\norm{\rot\ve_{\eps}\times\ve_{\eps}}_{L^2(0,\infty;L^1)\cap L^1(0,\infty;L^{\f32})}\leq C
		\end{equation}
		by the H\"older inequality. An interpolation then gives
		\begin{equation*}
			\norm{\rot\ve_{\eps}\times\ve_{\eps}}_{L^{r'}(0,\infty;L^{\f{3r}{2r+2}})}\leq C.
		\end{equation*}
		As $\f{3r}{2r+2}\geq\f98>1$ and $\X^r\embl L^r(0,\infty;L^9(\Omega;\R^3))$, we deduce that
		\begin{align}\label{rot}
			\norm{\rot\ve_{\eps}\times\ve_{\eps}}_{(\X^r)'}\leq C\norm{\rot\ve_{\eps}\times\ve_{\eps}}_{L^{r'}(0,\infty;L^{\f98})}\leq C.
		\end{align}
		On the other hand, if $\f{11}5\leq r<3$, we use instead the information that $\ve_{\eps}$ is bounded in $\X^r$, replacing \eqref{rotpr} with
		\begin{equation*}
			\norm{\rot\ve_{\eps}\times\ve_{\eps}}_{L^r(0,\infty;L^{\f{2r}{r+2}})\cap L^{\f r2}(0,\infty;L^{\f{3r}{6-r}})}\leq C.
		\end{equation*}
		Using interpolation once again yields
		\begin{equation*}
			\norm{\rot\ve_{\eps}\times\ve_{\eps}}_{L^{r'}(0,\infty;L^{z'})}\leq C,\quad\text{where}\quad z\coloneqq\f{6r}{(r-2)(5r-3)}.
		\end{equation*}
		A direct calculation verifies that, in the considered range of $r$, we have $1<z\leq\f{33}4\leq\f{3r}{3-r}$, hence the embeddings
		\begin{equation*}
			W^{1,r}(\Omega;\R^3)\embl L^z(\Omega;\R^3)\quad\text{and}\quad L^{r'}(0,\infty;L^{z'}(\Omega;\R^3))\embl (\X^r)',
		\end{equation*}
		hold true, from which we deduce
		\begin{equation*}
			\norm{\rot\ve_{\eps}\times\ve_{\eps}}_{(\X^r)'}\leq C\norm{\rot\ve_{\eps}\times\ve_{\eps}}_{L^{r'}(0,\infty;L^{z'})}\leq C.
		\end{equation*}
		This together with \eqref{rot} gives
		\begin{equation}\label{rotww}
			\norm{\rot\ve_{\eps}\times\ve_{\eps}}_{(\X^r)'}\leq C
		\end{equation}
		for any $r\geq\f{11}5$. 
		
		Next, we estimate $\doc{\ve}_{\eps}$ by applying a~similar method as in \cite{OSU}. Estimates \eqref{odhad1}--\eqref{odhad3} proved thus far and the H\"older inequality show that the functionals $N^1_{\eps}\ve_{\eps}$, $N^2_{\eps}\ve_{\eps}$, $A^1_{\eps}\ve_{\eps}$, $A^2_{\eps}\ve_{\eps}$ and $A\ve_{\eps}$, defined in Lemma~\ref{LEL}, are bounded in the following sense:
		\begin{align}
			\norm{N^1_{\eps}\ve_{\eps}}_{(\X^{\f{2q}{q-2}})'}&\leq C\norm{\eps^{\f1q}\ve_{\eps}}_{\X^q}\norm{\eps^{\f12}\partial_t\ve_{\eps}}_{L^2(0,\infty;L^2)}\leq C,\label{Nest1}\\
			\norm{N^2_{\eps}\ve_{\eps}}_{(\X^{\f q{q-3}})'}&\leq C\norm{\eps^{\f1q}\ve_{\eps}}^3_{\X^q}\leq C,\label{Nest2}\\
			\norm{A^1_{\eps}\ve_{\eps}}_{(\X^4)'}&\leq C\norm{\eps^{\f14}\ve_{\eps}}^3_{\X^4}\leq C,\label{Aest1}\\
			\norm{A^2_{\eps}\ve_{\eps}}_{(\X^q)'}&\leq C\norm{\eps^{\f1q}\ve_{\eps}}^{q-1}_{\X^q}\leq C,\label{Aest}\\
			\norm{A\ve_{\eps}}_{(\X^2\cap \X^r)'}&\leq C\norm{\ve_{\eps}}_{\X^2}+C\norm{\ve_{\eps}}_{\X^r}^{r-1}\leq C.\label{Aes}
		\end{align}
		Then, we use \eqref{EL3} to express $\doc{\ve}_{\eps}$ as a~temporal convolution with the kernel $K(t)\coloneqq\eps^{-1}e^{\f{t}{\eps}}\chi_{t\leq0}$. This leads to
		\begin{equation}\label{conv}
			\doc{\ve}_{\eps}=K*(\eps^{\f12-\f1q}N^1_{\eps}\ve_{\eps}+\eps^{1-\f3q}N^2_{\eps}\ve_{\eps}+\eps^{\f1q}A_{\eps}\ve_{\eps}+A\ve_{\eps}-\fe),
		\end{equation}
		which is understood as an identity in the space $L^1_{\rm loc}(0,\infty;V^{-1,q'}_{\di})$. Using the properties of convolution, \eqref{Nest1}--\eqref{Aes}, and recalling \eqref{ineq}, the right-hand side of \eqref{conv} is a~continuous linear functional in the space $\X^2\cap \X^q$. Identity \eqref{conv} thus gives
		\begin{equation}\label{docb}
			\norm{\doc{\ve}_{\eps}}_{(\X^2_{\di}\cap \X^q_{\di})'}\leq C
		\end{equation}
		and in combination with \eqref{rotww} and the embedding $\X^2_{\di}\cap \X^q_{\di}\embl \X^r$, this proves \eqref{odhad4}.
	\end{proof}
	
	It is clearly seen in the proof above that the assumption $r\geq\f{11}5$ is used only to show that the convective term is a~bounded functional on $\X^r$, uniformly with respect to $\eps>0$. This information is useful later when taking the limit $\eps\to0_+$. Otherwise, it is important that $r\geq2$, because then one can absorb the term $\ii(\nn\ve_{\eps})\ve_{\eps}\cdot\ve_0^{\eps}$, recall \eqref{rotp} and remark (vi) above.
	
	It is apparent that the term $\eps|\Dv_{\eps}|^4$ is needed to control terms related to the convective term. The role of the higher order stabilization $\eps|\Dv_{\eps}|^q$, $q>4$, is later clarified while taking the limit $\eps\to0_+$.
	
	Comparing with the usual existence theories for Navier-Stokes equations, one may wonder why we need a~super-linear growth also in the boundary terms on $\Gn$. We recall that we do not want to impose additional geometrical assumptions on $\Omega$. In this case, there will always be a~term on the right-hand side of \eqref{poin} that controls the overall speed of the flow (to have just $\norm{\we}_{1,p}\leq c'_p\norm{\D\we}_p$, one would need to exclude domains that are ``too special'' such as axisymmetric domains, parallel plates etc.) As opposed to usual existence theories, we cannot choose this term to be $(\int_{\Omega}|\we|)^p$ since we do not know that $\norm{\ve_{\eps}}_{L^{\infty}(0,\infty;L^2)}$ is bounded a~priori. In our case, this information needs to be carefully deduced from the Euler-Lagrange equations by testing with a~solution (minimum), but this generates many terms without a~sign, especially in the case with nonhomogeneous data, as can be seen in the proof above. It thus seems natural to take instead into consideration the fact that the fluid loses energy also due to friction on $\Gn$. But then \eqref{poin} indirectly requires the scalings of $\Sb$ and $\s$ to be compatible, explaining the same nonlinear growth. The whole situation would simplify in the case we considered $I_{\eps}$ for the full velocity gradient, since then one does not need the Korn inequality. Nevertheless, even for the Poincar\'e inequality to hold in the form $\norm{\we}_{1,p}\leq c_p''\norm{\nn\we}_p$, certain domains have to be ruled out.

        \subsection{\UUU Pressure reconstruction}
	At this point, the most of the work leading to Theorem~\ref{Tex} is done. Before proceeding with its proof, let us state one more auxiliary result, that is used in the pressure construction. Let us define $L^p_0(\Omega;\R)\coloneqq\{f\in L^p(\Omega;\R):\ii f=0\}$. Up to the boundary conditions, the following proposition is very standard.
	
	\begin{prop}\label{Prop1}
		Let $1<p<\infty$ and $\gb\in V^{-1,p'}$ be such that
		\begin{equation}\label{grad}
			\scal{\gb,\fit}=0\quad\text{for all }\fit\in V^{1,p}_{\di},
		\end{equation}
		Then, there exists an unique function $\qr\in L^{p'}_0(\Omega;\R)$ satisfying
		\begin{equation}\label{Sto2}
			\scal{\gb,\fit}=-\ii \qr\di\fit\quad\text{for all }\fit\in V^{1,p}
		\end{equation}
		and
		\begin{equation}\label{stoest}
			\norm{\qr}_{p'}\leq C(p,\Omega)\norm{\gb}_{V^{-1,p'}}.
		\end{equation}
	\end{prop}
	\begin{proof}
		Since $W^{1,p}_0(\Omega;\R^3)\subset V^{1,p}$, one can apply, e.g., the result \cite[III.5.1]{galdi} to get \eqref{Sto2}. The pressure estimate \eqref{stoest} can be found, e.g., in \cite[Corollary~2.5.]{Girault}.
		
		To see more explicitly that the boundary conditions encoded in $V^{1,p}$ do not cause any difficulties, one can show that the auxiliary problem
		\begin{equation}\label{Sto}
			\ii|\nn\uu|^{p-2}\nn\uu\cdot\nn\fit+\ii|\uu|^{p-2}\uu\cdot\fit-\ii \qr\di\fit=\scal{\gb,\fit}
		\end{equation}
		for all $\fit\in V^{1,p}$ admits an unique solution $(\uu,\qr)\in V^{1,p}\times L^{p'}_0(\Omega;\R)$ whenever $\gb\in V^{-1,p'}$, which then obviously gives \eqref{Sto2} if \eqref{grad} holds. This is nothing but the weak formulation of the nonlinear problem
		\begin{align*}
			\di\uu&=0,&-\di(|\nn\uu|^{p-2}\nn\uu)+|\uu|^{p-2}\uu+\nn \qr&=\gb&\text{in }\Omega,\\
			\uu&=0,&&&\text{on }\Gd,\\
			\uu\cdot\n&=0,&((-\qr\I+|\nn\uu|^{p-2}\nn\uu)\n)_{\tau}&=0&\text{on }\Gn,\\
			\int_{\Gf^i}\uu\cdot\n&=0,&(-\qr\I+|\nn\uu|^{p-2}\nn\uu)\n&=c_i\n&\text{on }\Gf^i,
		\end{align*}
		where $c_i$ are the Lagrange multipliers to the constraints $\int_{\Gf^i}\uu\cdot\n=0$, $i=0,\ldots,n$. To find a~solution to \eqref{Sto} one may proceed by minimizing the functional
		\begin{equation*}
			J_k(\uu)\coloneqq\f1p\ii(|\nn\uu|^p+|\uu|^p+k|\di\uu|^p)-\scal{\gb,\uu}
		\end{equation*}
		and letting $k\to\infty$ (cf.~\cite[Ch.~I, \S 6]{Temam}). In any case, to obtain the pressure estimate \eqref{stoest}, one has to verify the inf-sup condition
		\begin{equation*}
			\inf_{\qr\in L^{p'}_0(\Omega;\R)}\sup_{\fit\in V^{1,p}}\f{\ii \qr\di\fit}{\norm{\qr}_{p'}\norm{\fit}_{1,p}}\geq C>0.
		\end{equation*}
		This condition is again an immediate consequence of $W^{1,p}_0(\Omega;\R)\subset V^{1,p}$, the fact that the norms $\norm{\cdot}_{1,p}$ and $\norm{\nn\cdot}_p$ are equivalent on $W^{1,p}_0(\Omega;\R^3)$ and the standard inf-sup condition for the pair $(W^{1,p}_0(\Omega;\R^3),L^{p'}_0(\R))$ (to be found in various forms in the works by O.~A.~Ladyzhenskaya, \mbox{J.-L.~Lions}, E.~Magenes, I.~Babu\v{s}ka, J.~Ne\v{c}as, or F.~Brezzi), which can be proved by applying the Bogovskii operator to $|\qr|^{p'-2}\qr-\f1{|\Omega|}\ii|\qr|^{p'-2}\qr\in L^p_0(\Omega;\R)$.
	\end{proof}
	
	Note that Proposition~\ref{Prop1} works for a~general Lipschitz domain $\Omega$ (actually only the local cone property is needed), which is desirable in our application, cf.~Figure~\ref{fig}. This contrasts with other methods of constructing $\qr$, such as the Helmholtz decomposition or the $L^p$-theory for the Stokes system that require some regularity of $\partial\Omega$, cf.~\cite{Girault} and references therein.

        \subsection{\UUU Passage to the limit as $\eps\to 0_+$} 
                \UUU Let $\{\ve_{\eps}\}_{\eps>0}$ be a sequence of
                minimizers to $I_\eps$, which exist due to Lemmas~\ref{Lex} and
                ~\ref{LEL}. \EEE
		
	\UUU The \EEE uniform estimates \eqref{odhad1}--\eqref{odhad4} guaranteed by Lemma~\ref{Luni}, reflexivity of the underlying spaces and standard compactness arguments involving the compact Sobolev embeddings,  the Aubin-Lions lemma and the Vitali convergence theorem imply the existence of a~function $\ve$ with property \eqref{NSpro1} and of a~(not relabeled) subsequence of $\{\ve_{\eps}\}_{\eps>0}$, satisfying \eqref{cc1}--\eqref{cc7} and also
		\begin{align}
			&\rot\ve_{\eps}\times\ve_{\eps}\wc\rot\ve\times\ve&&\text{weakly in }L^{\f r2}(Q_{\infty};\R^3),\label{cc10}\\
			&\doc{\ve}_{\eps}\wc\doc{\ve}&&\text{weakly in }(\X^2_{\di}\cap\X^q_{\di})',\\
			&\eps|\Dv_{\eps}|^2\Dv_{\eps}\to0&&\text{strongly in }L^{\f43}(Q_{\infty};\R^3),\\
			&\eps|\Dv_{\eps}|^{q-2}\Dv_{\eps}\to0&&\text{strongly in }L^{q'}(Q_{\infty};\R^3),\\
			&\Sb(\Dv_{\eps})\wc S&&\text{weakly in }L^2(Q_{\infty};\Sym)+L^{r'}(Q_{\infty};\Sym),\\
			&\eps|\ve_{\eps}|^2\ve_{\eps}\to0&&\text{strongly in }L^{\f43}(\Sigma_{\infty};\R^3),\\
			&\eps|\ve_{\eps}|^{q-2}\ve_{\eps}\to0&&\text{strongly in }L^{q'}(\Sigma_{\infty};\R^3),\\
			&\s(\ve_{\eps})\wc\z&&\text{weakly in }L^2(\Sigma_{\infty};\R^3)+L^{r'}(\Sigma_{\infty};\R^3),\label{cc14}
		\end{align}
		for some functions $S$ and $\z$ and for $\Sigma_T\coloneqq(0,T)\times\Gn$, $T\in(0,\infty]$. These convergences are clearly sufficient to take the limit in \eqref{EL2}, yielding
		\begin{equation}\label{lim}
			\int_0^{\infty}\ii(-\ve\cdot\partial_t\fit+(\rot\ve\times\ve)\cdot\fit+S\cdot\D\fit)+\int_0^{\infty}\inn\z\cdot\fit=\int_0^{\infty}\ii\fe\cdot\fit
		\end{equation}
		for all $\fit\in \Cl_c^{\infty}((0,\infty);V^{1,q}_{\di})$. Recalling \eqref{rotww}, note that all terms, except for the time derivative, are well defined also if $\fit\in \X^2_{\di}\cap \X^r_{\di}$. Therefore, we read from \eqref{lim} that the functional $\partial_t\ve$ extends uniquely to $\partial_t\ve\in(\X^2_{\di}\cap \X^r_{\di})'$ (proving \eqref{NSpro2}), and hence
		\begin{equation}\label{wea}
			\int_0^{\infty}\scal{\partial_t\ve,\fit}+\int_{Q_{\infty}}((\rot\ve\times\ve)\cdot\fit+S\cdot\D\fit-\fe\cdot\fit)+\int_0^{\infty}\inn\z\cdot\fit=0
		\end{equation}
		for all $\fit\in \X^2_{\di}\cap \X^r_{\di}$.
		
		To identify $\ve(0)$, we recall that, by our construction, we have $\ve_{\eps}(0)=\ve_0^{\eps}(0)\to\uu_0$ strongly in $L^2(\Omega;\R^3)$. Further, as the sequence $\partial_t\ve_{\eps}$ is uniformly bounded in $L^{q'}(0,T;V^{-1,q'}_{\di})$ for some $T>0$, there is a~non relabeled subsequence $\ve_{\eps}$ converging strongly in $\Cl([0,T];V^{-1,q'}_{\di})$ by the Arzel\`a-Ascoli theorem. In particular, we have $\ve_{\eps}(0)\to\ve(0)$ in $V^{-1,q'}_{\di}$, and hence $\ve(0)=\uu_0$ (as both $\ve(0)$ and $\uu_0$ are divergence-free), which is \eqref{NSpro3}.
		
		By the properties of the trace operator (see \cite[Corollary~1.13.]{Indiana}), it is standard to show that the trace of $\ve_{\eps}$ actually converges strongly to the trace of $\ve$ on $\Gn$ (proving the second part of \eqref{cc4}) and then, by the continuity of $\s$, this necessarily means that $\z=\s(\ve)$.
		
		To prove \eqref{NSdiv}, it remains to identify the weak limit $S$. To this end, we take advantage of the fact that in the considered case $r\geq\f{11}5$, the function $\ve$, after a~correction of boundary values, is an admissible test function in \eqref{wea}. Let $\ve_{\delta}$ be a~fixed element of the approximating sequence $\{\ve_{\eps}\}_{\eps>0}$ and let $0\leq\eta\in\Cl^{\infty}_c((0,T))$, $T>0$. Next, we observe that
		\begin{align}
			h\coloneqq&\int_0^T\Big(\scal{\doc{\ve},\ve_{\delta}}+\ii\big(S\cdot\D\ve_{\delta}-\fe\cdot\ve_{\delta}\big)+\inn\s(\ve)\cdot\ve_{\delta}\Big)\eta
		\end{align}
		and
		\begin{align}
			h_{\eps}\coloneqq&\int_0^T\Big(\ii\big(\doc{\ve}_{\eps}\cdot(\ve_{\delta}+\eps\partial_t\ve_{\delta}+\eps\rot\ve_{\eps}\times\ve_{\delta}+\eps\rot\ve_{\delta}\times\ve_{\eps})\nonumber\\
			&+\Sb_{\eps}(\Dv_{\eps})\cdot\D\ve_{\delta}-\fe\cdot\ve_{\delta}\big)+\inn\s_{\eps}(\ve_{\eps})\cdot\ve_{\delta}\Big)\eta+\eps\int_0^T\ii\doc{\ve}_{\eps}\cdot\ve_{\delta}\partial_t\eta
		\end{align}
		are well-defined and finite quantities. Moreover, using the convergence results \eqref{cc1}--\eqref{cc7}, \eqref{cc10}--\eqref{cc14}, Lemma~\ref{Luni} and the property $\ve_{\delta}\in\ve_0^{\delta}+\X^2_{\di}\cap \X^q_{\di}$, it is not hard to show
		\begin{equation}\label{hs}
			h_{\eps}\to h\quad\text{as}\quad\eps\to0_+.
		\end{equation}
		Further, using $\fit\coloneqq(\ve-\ve_{\delta})\eta\in \X^2_{\di}\cap \X^r_{\di}$ as a~test function in \eqref{wea} leads to
		\begin{align}\label{SDv}
			&\int_{Q_T}(-\tfrac12|\ve|^2\partial_t\eta+S\cdot\Dv\eta-\fe\cdot\ve\eta)+\int_{\Sigma_T}\s(\ve)\cdot\ve\eta=h.
		\end{align}
		Next, we use $\pit\coloneqq(\ve_{\eps}-\ve_{\delta})\eta\in \X^2_{\di}\cap \X^q_{\di}$ in \eqref{EL2}, \eqref{a+b}, Young's and H\"older's inequalities, $q>4$, and \eqref{rotw0}, giving
		\begin{align*}
			&\int_{Q_T}(-\tfrac12|\ve_{\eps}|^2\partial_t\eta+\Sb_{\eps}(\Dv_{\eps})\cdot\D\ve_{\eps}\eta-\fe\cdot\ve_{\eps}\eta)+\int_{\Sigma_T}\s_{\eps}(\ve_{\eps})\cdot\ve_{\eps}\eta\\
			&\quad=-\eps\int_{Q_T}(\partial_t\ve_{\eps}\cdot\partial_t(\ve_{\eps}\eta)+(\rot\ve_{\eps}\times\ve_{\eps})\cdot\partial_t(\ve_{\eps}\eta)\\
			&\quad\qquad+2\partial_t\ve_{\eps}\cdot(\rot\ve_{\eps}\times\ve_{\eps})\eta+2|\rot\ve_{\eps}\times\ve_{\eps}|^2\eta)+h_{\eps}\\
			&\quad=\eps\int_{Q_T}\!\!\!\!\big(\tfrac12|\ve_{\eps}|^2\partial_{tt}^2\eta-|\partial_t\ve_{\eps}|^2\eta-3\partial_t\ve_{\eps}\cdot(\rot\ve_{\eps}\times\ve_{\eps})\eta-2|\rot\ve_{\eps}\times\ve_{\eps}|^2\eta\big)+h_{\eps}\\
			&\quad\leq \eps\int_{Q_T}\big(\tfrac12|\ve_{\eps}|^2\partial_{tt}^2\eta+\tfrac14|\rot\ve_{\eps}\times\ve_{\eps}|^2\eta\big)+h_{\eps}\\
			&\quad\leq C(\eta)(\eps\norm{\ve_{\eps}}_{\X^2}^2+\eps^{1-\f4q}\norm{\eps^{\f1q}\ve_{\eps}}_{\X^4}^4)+h_{\eps}\\
			&\quad\leq C(\eta,T)(\eps+\eps^{1-\f4q})+h_{\eps}.
		\end{align*}
		Now, we take the limes superior of this inequality and on the left-hand side we use that $\ve_{\eps}\to\ve$ strongly in $L^2(Q_T;\R^3)$ (by interpolation and Vitali's theorem), the inequality $\Sb_{\eps}(\Dv_{\eps})\cdot\Dv_{\eps}\geq\Sb(\Dv_{\eps})\cdot\Dv_{\eps}$, and in the boundary term we use \eqref{cc4} and Fatou's lemma. This way, we get
		\begin{align*}
			h&\geq\limsup_{\eps\to0_+}\Big(\int_{Q_T}(-\tfrac12|\ve_{\eps}|^2\partial_t\eta+\Sb_{\eps}(\Dv_{\eps})\cdot\D\ve_{\eps}\eta-\fe\cdot\ve_{\eps}\eta)+\int_{\Sigma_T}\s_{\eps}(\ve_{\eps})\cdot\ve_{\eps}\eta\Big)\\
			&\geq\limsup_{\eps\to0_+}\int_{Q_T}\Sb(\Dv_{\eps})\cdot\D\ve_{\eps}\eta+\lim_{\eps\to0_+}\int_{Q_T}(-\tfrac12|\ve_{\eps}|^2\partial_t\eta-\fe\cdot\ve_{\eps}\eta)\\
			&\qquad+\liminf_{\eps\to0_+}\int_{\Sigma_T}\s(\ve_{\eps})\cdot\ve_{\eps}\eta\\
			&\geq\limsup_{\eps\to0_+}\int_{Q_T}\Sb(\Dv_{\eps})\cdot\D\ve_{\eps}\eta+\int_{Q_T}(-\tfrac12|\ve|^2\partial_t\eta-\fe\cdot\ve\eta)+\int_{\Sigma_T}\s(\ve)\cdot\ve\eta
		\end{align*}
		Comparing this with \eqref{SDv} immediately leads to
		\begin{equation*}
			\limsup_{\eps\to0_+}\int_{Q_T}\Sb(\Dv_{\eps})\cdot\Dv_{\eps}\eta\leq\int_{Q_T}S\cdot\Dv\eta.
		\end{equation*}
		Hence, by the monotonicity of $\Sb$, we get, for any $W\in L^r(Q_T;\Sym)$, that
		\begin{align*}
			0&\leq\limsup_{\eps\to0_+}\int_{Q_T}(\Sb(\Dv_{\eps})-\Sb(W))\cdot(\Dv_{\eps}-W)\eta\\
			&\leq\int_{Q_T}(S\cdot\Dv-S\cdot W-\Sb(W)\cdot\Dv+\Sb(W)\cdot W)\eta\\
			&=\int_{Q_T}(S-\Sb(W))\cdot(\Dv-W)\eta.
		\end{align*}
		Choosing now $W=\Dv+\lambda U$, $U\in L^r(Q_T;\Sym)$ and dividing by $\lambda>0$ yields
		\begin{equation*}
			0\leq\int_{Q_T}(\Sb(\Dv+\lambda U)-S)\cdot U\eta
		\end{equation*}
		and, consequently, using the continuity of $\Sb$ to take the limit $\lambda\to0_+$, we arrive at
		\begin{equation*}
			0\leq\int_{Q_T}(\Sb(\Dv)-S)\eta\cdot U.
		\end{equation*}
		Since $U$ is arbitrary, we deduce that $\Sb(\Dv)\eta=S\eta$ a.e.\ in $Q_T$, but since $\eta$ and $T$ are also arbitrary, we conclude that $\Sb(\Dv)=S$ a.e.\ in $Q_{\infty}$ and \eqref{NSdiv} is proved.
		
		In the next step, we prove \eqref{P} by constructing a~pressure in \eqref{NSdiv}. Since the test functions from \eqref{NSdiv} must vanish on $\Gd$, we follow the same construction of pressure as in \cite[Theorem~2.6.]{Wolf}, but only partially, since we do not need a~pressure decomposition here. 
		
		We fix $t\in(0,\infty)$ and choose $\fit=\chi_{(-\infty,t]}\pit$ in \eqref{wea} to get
		\begin{equation}
			\scal{\gb(t),\pit}=0\quad\text{for all}\quad\pit\in\Cl^{\infty}_{\partial,\di},\label{pres}
		\end{equation}
		where
		\begin{align}
			\scal{\gb(t),\pit}&\coloneqq\ii\ve(t)\cdot\pit+\ii\int_0^t(\rot\ve\times\ve)\cdot\pit-\ii\int_0^t\fe\cdot\pit\nonumber\\
			&\quad+\ii\int_0^t\Sb(\Dv)\cdot\D\pit+\inn\int_0^t\s(\ve)\cdot\pit,\quad\pit\in V^{1,r}.\label{gdef}
		\end{align}
		As $V^{1,r}\embl V^{1,2}\embl V^{1,\f32}\embl V^{1,\f65}\embl L^2(\Omega;\R^3)$ and $V^{1,\f32}\embl L^3(\Omega;\R^3)$, the functional $\gb$ can be estimated using \eqref{NSpro1} as
		\begin{align*}
			\norm{\gb(t)}_{V^{-1,r'}}&\leq C\sup_{\norm{\fit}_{1,r}\leq 1}\!\!\Big(\norm{\ve(t)}_2\norm{\fit}_2+\int_0^t\!\norm{\rot\ve}_2\norm{\ve}_6\norm{\fit}_3+\int_0^t\norm{\fe}_2\norm{\fit}_2\\
			&\qquad+\int_0^t\!\norm{\Sb(\Dv)}_{r'}\norm{\D\fit}_{r}+\int_0^t\norm{\s(\ve)}_{r';\Gn}\norm{\fit}_{r;\Gn}\Big)\\
			&\leq C\norm{\ve(t)}_2+C\int_0^t(\norm{\ve}_{1,2}^2+\norm{\fe}_2+\norm{\Sb(\Dv)}_{r'}+\norm{\s(\ve)}_{r';\Gn})\\
			&\leq C(1+t^{\f12}).
		\end{align*}
		Consequently, the relation \eqref{pres} holds also for all $V^{1,r}_{\di}$ and thus, by Proposition~\ref{Prop1}, there exists a~unique function $Q(t)\in L^{r'}_0(\Omega;\R)$ satisfying
		\begin{equation}\label{Stokes}
			-\ii Q(t)\di\fit=\scal{\gb,\fit}\quad\text{for all }\fit\in V^{1,r}
		\end{equation}
		and
		\begin{equation*}
			\norm{Q(t)}_{r'}\leq C(1+t^{\f12}),
		\end{equation*}
		showing that $Q\in L^{\infty}_{\rm loc}(0,\infty;L^{r'}_0(\Omega;\R))$ (the Bochner measurability of $Q$ is a~consequence of the weak continuity of $\ve$ in time). Further, we infer from \eqref{gdef} and \eqref{Stokes} that
		\begin{align}
			&-\int_0^{\infty}\ii\ve\cdot\partial_t\pit+\int_0^{\infty}\Big(\ii(\rot\ve\times\ve)\cdot\pit+\ii\Sb(\D\ve)\cdot\D\pit+\inn\s(\ve)\cdot\pit\Big)\nonumber\\
			&\qquad\qquad\qquad\qquad\qquad
			=\int_0^{\infty}\ii\fe\cdot\pit+\int_0^{\infty}\ii Q\di\partial_t\pit\label{Q}
		\end{align}
		for all $\pit\in\Cl^{\infty}_c((0,\infty);V^{1,r})$. Next, we define the function
		\begin{equation*}
			K(t,x)\coloneqq\int_0^t\tfrac12|\ve(s,x)|^2\dd{s}
		\end{equation*}
		and note, using properties of the Bochner integral, that
		\begin{equation*}
			2\esssup_{(0,T)}\norm{K}_3=\Big\Vert\int_0^{T}|\ve|^2\Big\Vert_3\leq\int_0^T\norm{\ve}_6^2\leq\norm{\ve}^2_{\X^2}\leq C
		\end{equation*}
		for all $T>0$, hence $K\in L^{\infty}(0,\infty;L^3(\Omega;\R))$. Next, integration by parts shows
		\begin{align*}
			-\int_0^{\infty}\ii K\di\partial_t\pit&=\int_0^{\infty}\ii\tfrac12|\ve|^2\di\pit\\
			&=\int_0^{\infty}\int_{\Gf}\tfrac12|\ve|^2\pit\cdot\n-\int_0^{\infty}\ii\nn(\tfrac12|\ve|^2)\cdot\pit,
		\end{align*}
		where we used that $\nn(\tfrac12|\ve|^2)=(\nn\ve)^T\ve$ is summable in $Q_{\infty}$ and that
		\begin{equation*}
			\ve\in L^2(0,\infty;W^{\f12,2}(\partial\Omega;\R^3))\embl L^2(0,\infty;L^4(\partial\Omega;\R^3)).
		\end{equation*}
		Adding this to \eqref{Q}, recalling \eqref{dotu} and defining
		\begin{equation*}
			P_0\coloneqq Q+K\in L^{\infty}_{\rm loc}(0,\infty;L^{r'}(\Omega;\R))
		\end{equation*}
		leads to \eqref{P} with $P_0$ instead of $P$. Finally, we remark that the form of \eqref{P} remains unchanged if the pressure is shifted by a~function $E\in L^{\infty}_{\rm loc}(0,\infty;\R)$ of time only. Indeed, this is a~consequence of 
		\begin{equation*}
			\int_0^{\infty}\ii E\di\partial_t\pit=\int_0^{\infty}\!\!\!E\,\partial_t\int_{\partial\Omega}\pit\cdot\n=0\quad\text{for all }\pit\in\Cl^{\infty}_c((0,\infty);V^{1,r}).
		\end{equation*}
		The choice $P\coloneqq P_0-\f1{|\Omega|}\ii K+D$ then leads precisely to \eqref{Ppr} and \eqref{P}.

                \subsection{\UUU Identification of boundary conditions}
		To conclude the proof, it remains to identify the boundary conditions on $\Gn\cup\Gf$, that are encoded implicitly in \eqref{P}. Let us choose $\pit=\psi\fit$ with $\fit\in V^{1,r}$ and $\psi\in\Cl^1_c((0,\infty);\R)$ fixed, use the fact that $S\cdot\D\fit=S\cdot\nn\fit$ whenever $S$ is a~symmetric matrix, and rewrite \eqref{P} as
		\begin{align}\label{cas}
			\int_{\Omega}\T_{\psi}\cdot\nn\fit&=\int_{\Omega}\int_0^{\infty}(\ve\partial_t\psi-(\nn\ve)\ve\psi+\fe\psi)\cdot\fit\nonumber\\
			&\qquad+\int_{\Gf}\int_0^{\infty}\tfrac12|\ve|^2\n\psi\cdot\fit-\inn\int_0^{\infty}\s(\ve)\psi\cdot\fit.
		\end{align}
		In particular, by choosing $\fit$ with compact support in $\Omega$ and taking into consideration that $\nn\ve\in L^r(0,\infty;L^r(\Omega;\R^3))\embl L^r(0,\infty;L^2(\Omega;\R^3))$ and $\ve\in\X^r_{\di}\embl L^r(0,\infty;L^6(\Omega;\R^3))$ imply $(\nn\ve)\ve\in L^{\f r2}(0,\infty;L^{\f32}(\Omega;\R^3))$, we read from \eqref{cas} that
		\begin{equation*}
			\di\T_{\psi}\in L^{\f32}(\Omega;\R^3).
		\end{equation*}
		Hence, we can define a~continuous linear functional 
		\begin{equation}\label{Tnp}
			\T_{\psi}\n\in (W^{\f1{r'},r}(\partial\Omega;\R^d))'
		\end{equation}
		by the formula
		\begin{equation*}
			\scal{\T_{\psi}\n,\we}\coloneqq\ii(\di\T_{\psi}\cdot E\we+\T_{\psi}\cdot\nn E\we)\quad\text{for all }\we\in W^{\f1{r'},r}(\partial\Omega;\R^d),
		\end{equation*}
		where $E:W^{\f1{r'},r}(\partial{\Omega};\R^3)\to W^{1,r}(\Omega;\R^3)$ is the continuous linear trace-extension operator (inverse of the trace operator). Then, integration by parts in \eqref{cas} yields
		\begin{align}
			&\int_{\Omega}\Big(\di\T_{\psi}-\int_0^{\infty}(\ve\partial_t\psi-(\nn\ve)\ve\psi+\fe\psi)\Big)\cdot\fit\nonumber\\
			&\quad=\langle\T_{\psi}\n,\fit\rangle_{\Gf}-\int_{\Gf}\int_0^{\infty}\tfrac12|\ve|^2\n\psi\cdot\fit\nonumber\\
			&\qquad+\langle\T_{\psi}\n,\fit\rangle_{\Gn}+\inn\int_0^{\infty}\s(\ve)\psi\cdot\fit.\label{rov}
		\end{align}
		for all $\fit\in V^{1,r}$ and $\psi\in\Cl^1_c((0,\infty);\R)$. Since this is true in particular for every $\fit\in\Cl^{\infty}_c(\Omega;\R^3)$, we recover the Navier-Stokes equation in the form
		\begin{equation*}
			\di\T_{\psi}=\int_0^{\infty}(\ve\partial_t\psi-(\nn\ve)\ve\psi+\fe\psi)\quad\text{a.e.\ in }\Omega,
		\end{equation*}
		and, consequently, returning to \eqref{rov}, also
		\begin{align}
			&\langle\T_{\psi}\n,\fit\rangle_{\Gf}-\int_{\Gf}\int_0^{\infty}\tfrac12|\ve|^2\n\psi\cdot\fit\nonumber\\
			&\qquad+\langle\T_{\psi}\n,\fit\rangle_{\Gn}+\inn\int_0^{\infty}\s(\ve)\psi\cdot\fit=0\quad\text{for all}\quad\fit\in V^{1,r}.\label{varbc}
		\end{align}		
		Let $\we\in W^{1,\infty}_{0,\n}(\Gn;\R^3)$ and extend it by zero to $\partial\Omega$. Then, we have $\we\in W^{1,\infty}(\partial\Omega;\R^3)$ and $\int_{\partial\Omega}\we\cdot\n=0$. Therefore, $\we$ can be extended to a~Lipschitz function in $\Omega$ such that $\we\in V^{1,r}$. Hence, \eqref{varbc} yields
		\begin{equation*}
			\langle\T_{\psi}\n,\we\rangle_{\Gn}+\inn\int_0^{\infty}\s(\ve)\psi\cdot\we=0\quad\text{for all}\quad\we\in W^{1,\infty}_{0,\n}(\Gn;\R^3).
		\end{equation*}
		Recalling \eqref{Tnp} and \eqref{LM0n}, this remains valid for all $\we\in W^{\f1{r'},r}_{0,\n}(\Gn;\R^3)$, proving \eqref{bc3}. Next, let $0\leq i\leq n$ and choose $\we\in W^{1,\infty}_0(\Gf^i;\R^3)$ such that $\int_{\Gf^i}\we\cdot\n=0$, extend this function by zero to whole $\partial\Omega$, and use \eqref{varbc} to deduce
		\begin{equation}
			\scal{G,\we}=0\quad\text{for all}\quad\we\in W^{1,\infty}_0(\Gf^i;\R^3)\quad\text{with}\quad\int_{\Gf^i}\we\cdot\n=0,\label{ide}
		\end{equation}
		where we abbreviated
		\begin{equation*}
			\scal{G,\we}\coloneqq \langle\T_{\psi}\n,\we\rangle_{\Gf^i}-\int_{\Gf^i}\int_0^{\infty}\tfrac12|\ve|^2\n\psi\cdot\we.
		\end{equation*}
		The restriction of $\we$ on $\Gf^i$ is a~consequence of the prescribed net fluxes and of the incompressibility of $\ve$, recall the definition of $V^{1,r}$. Since $\Gf^i$ is locally a~graph of a~Lipschitz function, there exists a~vector field $\etb_i\in W^{1,\infty}(\Gf^i;\R^3)$ such that $\etb_i\cdot\n>0$ on $\Gf^i$. Moreover, it is clear that $\etb_i$ can be chosen in a~way that $\etb_i\in W^{1,\infty}_0(\Gf^i;\R^3)$ and $\int_{\Gf^i}\etb_i\cdot\n=1$. Let $\z\in W^{1,\infty}_0(\Gf^i;\R^3)$. Then, the function
		\begin{equation*}
			\we\coloneqq\z-\etb_i\int_{\Gf^i}\z\cdot\n
		\end{equation*}
		can be used in \eqref{ide}, leading to
		\begin{equation*}
			\scal{G,\z}=\scal{G,\etb_i}\int_{\Gf^i}\z\cdot\n=\int_{\Gf^i}\scal{G,\etb_i}\n\cdot\z=\int_{\Gf^i}(c_i)_{\psi}\n\cdot\z.
		\end{equation*}
		Since $\z\in W^{1,\infty}_0(\Gf^i;\R^3)$ was arbitrary
                and $\n\in L^{\infty}(\Gf^i;\R^3)$, we deduce
                \eqref{bc2} by virtue of \eqref{LM0} and \eqref{Tnp}.  
        \qed
	
	The bounds $r\geq\f{11}5$ and $q>4$ are evidently used only to
        identify that the weak limit of $\Sb_{\eps}(\Dv_{\eps})$ is
        $\Sb(\Dv)$. The assumption $r\geq\f{11}5$ greatly simplifies
        the identification procedure since then the function $\ve$ can
        be used (after minor corrections) in \eqref{NSdiv} as a~test
        function. But it is unlikely that this bound for $r$ is
        necessary since in the mathematical theory of non-Newtonian
        fluids, more refined arguments are known for the limit
        identification. Unfortunately, the application of the methods
        of either \cite{Wolf}, or \cite{Sebastian}  to our
        $\eps$-approximation scheme seems not
        straightforward. Therefore, the case $\f65<r<\f{11}5$,
        $r\neq2$ is left open. However, this drawback seems not so
        significant in the view of the fact that our method works for
        $p$-fluids with $p$ sufficiently large that can approximate
        a~flow of an $r$-fluid for arbitrary $r>\f65$, see the proof
        of \cite[Theorem~3.1.]{Sebastian}, effectively avoiding the
        $\eps$-approximation.

        \UUU
\section*{Acknowledgements}  
This research is supported by the Austrian Science Fund (FWF) projects
F\,65, W\,1245,  I\,4354, I\,5149, and P\,32788 and by the OeAD-WTZ project CZ 01/2021.
        \EEE

	\bibliographystyle{siam}

\end{document}